\documentclass[11pt,letterpaper]{amsart}

\usepackage{amssymb, amsfonts, amsmath, amsthm,bm, amscd}
\usepackage[dvipsnames]{xcolor}
\usepackage{mathrsfs}
\usepackage{graphicx}
\usepackage{caption}
\usepackage{subcaption}
\usepackage{wrapfig}
\usepackage{enumitem, multicol}
\usepackage{tikz}
\usetikzlibrary{cd}
\usepackage{nicefrac, bigints}
\usepackage{array}
\usepackage{stmaryrd}
\usepackage{mathtools} 
\usepackage{extarrows} 
\usepackage{amsbsy}
\usepackage{bbm}
\usepackage[colorlinks,citecolor=cyan,linkcolor=teal]{hyperref}

\theoremstyle{definition}
\newtheorem{theorem}{Theorem}[section]

\newtheorem{definition}[theorem]{Definition}
\newtheorem{lemma}[theorem]{Lemma}

\newtheorem{corollary}[theorem]{Corollary}
\newtheorem{proposition}[theorem]{Proposition}
\newtheorem{remark}[theorem]{Remark}

\newtheorem{claim}[theorem]{Claim}

\newtheorem{conjecture}[theorem]{Conjecture}

\newtheorem*{notation}{Notation}

\makeatletter
\DeclareRobustCommand{\cev}[1]{%
  {\mathpalette\do@cev{#1}}%
}
\newcommand{\do@cev}[2]{%
  \vbox{\offinterlineskip
    \sbox\z@{$\m@th#1 x$}%
    \ialign{##\cr
      \hidewidth\reflectbox{$\m@th#1\vec{}\mkern4mu$}\hidewidth\cr
      \noalign{\kern-\ht\z@}
      $\m@th#1#2$\cr
    }%
  }%
}
\makeatother

\newcommand{\ZZ}{\mathbb{Z}}
\newcommand{\RR}{\mathbb{R}}
\newcommand{\QQ}{\mathbb{Q}}
\newcommand{\CC}{\mathbb{C}}

\newcommand{\inverse}{^{-1}}

\newcommand{\cN}{\mathcal{N}}

\newcommand{\cP}{\mathcal{P}}

\DeclareMathOperator{\Mod}{Mod}

\newcommand{\MF}{\mathcal{MF}}

\newcommand{\ML}{\mathcal{ML}}

\newcommand{\GL}{\mathsf{GL}}

\newcommand{\SL}{\mathsf{SL}}
\DeclareMathOperator{\im}{im}

\DeclareMathOperator{\vol}{vol}

\newcommand{\T}{\mathcal{T}}
\newcommand{\M}{\mathcal{M}}

\newcommand{\cH}{\mathcal{H}}

\newcommand{\PT}{\mathcal{PT}}
\newcommand{\PM}{\mathcal{PM}}
\newcommand{\PoT}{\mathcal{P}^1\mathcal{T}}
\newcommand{\PoM}{\mathcal{P}^1\mathcal{M}}
\newcommand{\QT}{\mathcal{QT}}
\newcommand{\QM}{\mathcal{QM}}
\newcommand{\QoT}{\mathcal{Q}^1\mathcal{T}}
\newcommand{\QoM}{\mathcal{Q}^1\mathcal{M}}
\newcommand{\sing}{\underline{\kappa}}

\DeclareMathOperator{\Sp}{\mathsf{Sp}}

\DeclareMathOperator{\PSL}{\mathsf{PSL}}

\DeclareMathOperator{\Area}{Area}

\newcommand{\sigl}{\sigma_\lambda}

\newcommand{\cO}{\mathcal{O}}

\DeclareMathOperator{\Ol}{\mathcal{O}_\lambda}

\newcommand{\arc}{\underline{\alpha}}

\DeclareMathOperator{\Stab}{Stab}

\DeclareMathOperator{\Eq}{Eq}

\DeclareMathOperator{\WP}{WP}
\DeclareMathOperator{\Th}{Th}

\renewcommand{\Re}{\operatorname{Re}}
\renewcommand{\Im}{\operatorname{Im}}

\newcommand{\TT}{\mathbb{T}}

\newcommand{\cQ}{\mathcal{Q}}

\newcommand{\cL}{\mathcal L}
\DeclareMathOperator{\hol}{hol}

\colorlet{lgray}{gray!40}

\DeclareMathOperator{\Dil}{Dil}

\DeclareMathOperator{\odd}{odd}
\DeclareMathOperator{\Mirz}{Mirz}

\newcommand{\AIS}{\mathcal L}

\newcommand{\TRG}{\mathcal{TRG}}

\newcommand{\rg}{\mathsf{Y}}

\newcommand{\singodd}{\sing_{\odd}}

\newcommand{\bfl}{\boldsymbol\ell}
\newcommand{\bfh}{\mathbf{h}}

\newcommand{\Addresses}{{
  \bigskip
  \footnotesize
  \noindent Aaron Calderon, \textsc{Department of Mathematics, University of Chicago}\par\nopagebreak
  \textit{E-mail address}: \texttt{aaroncalderon@uchicago.edu}
  
  \noindent James Farre, \textsc{Max Planck Institute for Mathematics in the Sciences, Leipzig}\par\nopagebreak
  \textit{E-mail address}: \texttt{james.farre@mis.mpg.de}
  }}

\begin{document}

\title{On Mirzakhani's twist torus conjecture}
\author{Aaron Calderon}
\author{James Farre}

\setcounter{tocdepth}{1}

\maketitle
\vspace{-5ex}
\thispagestyle{empty}

\begin{abstract}
We address a conjecture of Mirzakhani about the statistical behavior of certain expanding families of ``twist tori'' in the moduli space of hyperbolic surfaces, showing that they equidistribute to a certain Lebesgue-class measure along almost all sequences.
We also identify a number of other expanding families of twist tori whose limiting distributions are mutually singular to Lebesgue.
\end{abstract}

\section{Introduction}\label{sec:intro}

Let $\M_g$ be the moduli space of hyperbolic structures on a closed, oriented surface $S$ of genus $g\ge2$.
Let $\gamma$ be a \textbf{pants decomposition} of $S$, i.e., a maximal collection of non-isotopic, pairwise disjoint, essential, simple closed curves.
For any $L>0$, there is a \textbf{twist torus} $\TT_\gamma(L) \subset \M_g$ consisting of those hyperbolic surfaces glued together from pairs of pants with all boundary components of length $L$ along a pants decomposition of topological type $\gamma$.
The gluing/twist parameters take values in a torus $(\RR/L\ZZ)^{3g-3}$, as one full twist about any curve in $\gamma$ yields an isometric hyperbolic surface.
The twist torus $\TT_\gamma(L)$ is an immersed finite quotient\footnote{As $\gamma$ may have a  non-trivial finite symmetry group.} of $(\RR/L\ZZ)^{3g-3}$, which we equip with the pushforward $\tau_\gamma(L)$ of the Lebesgue measure on $(\RR/L\ZZ)^{3g-3}$.
See \S\ref{sec: hyp twist tori} for a more detailed discussion.

Mirzakhani conjectured\footnote{See \cite[Problem 13.2]{Wright_Mirz}.} that as $L \to \infty$, these  tori equidistribute in $\M_g$ to a certain  measure in the class of Lebesgue.
See \S\ref{sec:background} for definitions of the function $B(X)$, the constant $b_g$, and the measure $\vol_{\WP}$.

\begin{conjecture}[Mirzakhani's twist torus conjecture]\label{conj:twist tori}
Let $\gamma$ be any pants decomposition. Then
\[ \lim_{L \to \infty} \frac{\tau_\gamma(L)}{L^{3g-3}} = \frac{B(X )}{b_g} d \vol_{\WP}(X)\]
with respect to the weak-$\ast$ topology on measures on $\M_g$.
\end{conjecture}

This should be compared with a classical result regarding the equidistribution of expanding closed horocycles in the unit tangent bundle $\SL_2\ZZ \backslash \SL_2\RR = T^1Y$ of the modular curve $Y$.
Denote by $\Lambda(L)$ the Lebesgue measure on the (unique) closed orbit of the unstable horocycle flow of period $L$.
Then $\Lambda(L)/L$ converges weak-$*$ as $L\to \infty$ to the Liouville measure on $T^1Y$, normalized to have unit volume \cite{Sarnak}; see \cite[\S II]{EskMcM} for a proof using ideas introduced in Margulis's thesis \cite{Margulisthesis, Margulis:thesisEnglish} and mixing of the geodesic flow \cite{Hedlund}.
Using the projection $p: T^1Y\to Y$, one obtains an equidistribution result of the form 
\[\lim_{L \to \infty} \frac{p_* \Lambda(L)}{L}  = \frac{2\pi}{\vol(T^1Y)}d\Area_Y,
\]
where $d\Area_Y$ is the hyperbolic area measure on $Y$.
\medskip

Conjecture \ref{conj:twist tori} also has a refinement in terms of the equidistribution of expanding orbits of a flow.
Using geodesic length functions, Thurston identified the unit cotangent bundle of $\M_g$ with $\PoM_g$, the bundle of unit length measured geodesic laminations over $\M_g$ (see \S\ref{sec:background}).
This bundle carries the unipotent-like {\bf earthquake flow}: given a hyperbolic surface $X$ and a measured lamination $\lambda$, the earthquake flow acts by cutting open $X$ along $\lambda$, shearing along $\lambda$ according to its measure, and then regluing.
There is an earthquake flow--invariant ergodic measure $\mu_{\Mirz}$ on $\PoM_g$ in the class of Lebesgue, which satisfies
\[dp_* \mu_{\Mirz}(X) = {B(X)}~d\vol_{\WP}(X),\]
where $p: \PoM_g\to \M_g$ forgets the measured lamination. Its mass is the constant $b_g$.
Recent results have tied the ergodic theory of the earthquake flow with respect to $\mu_{\Mirz}$ to questions about counting hyperbolic geodesics \cite{Mirz_horo, AH_compcount, Liu_horo, spine}.

Twist tori lift to $\PoM_g$ by keeping track of the pants decomposition $\gamma$ weighted by $1/(3g-3)L$.
Evidently, these lifted twist tori are invariant under the earthquake flow.
Conjecture \ref{conj:twist tori} would therefore follow from the stronger statement that the normalized Lebesgue measures on lifted twist tori equidistribute to $\mu_{\Mirz}/b_g$ in $\PoM_g$ as $L \to \infty$.
\medskip

In this paper, we address this stronger version as well as the following generalization.
Fix a pants decomposition $\gamma = \gamma_1 \cup \ldots \cup \gamma_{3g-3}$, a length vector $\bfl \in \RR_{>0}^{3g-3}$, and a weight vector $\bfh \in \RR^{3g-3}_{>0}$ such that $\bfl \cdot \bfh = 1$.
Denote by $\bfh \gamma$ the measured lamination $h_1 \gamma_1 \cup \ldots \cup h_{3g-3}\gamma_{3g-3}$.
Setting the length of each $\gamma_i$ equal to $\ell_i$ and taking arbitrary twists yields an immersion of a (finite quotient of a) torus into $\M_g$ as before, and we can lift this this to a map
\[(\RR / \ell_1 \ZZ) \times \ldots \times (\RR / \ell_{3g-3} \ZZ)
\looparrowright \PoM_g\]
by specifying the lamination coordinate to be $\bfh \gamma$.
Let $\cP\TT_\gamma(\bfl, \bfh) \subset \PoM_g$ denote the image of this map and let $\mu_{\gamma}(\bfl, \bfh)$ be the pushforward of Lebesgue measure, normalized to have unit mass.
We observe that for almost every choice of weights $\bfh$, the measure $\mu_{\gamma}(\bfl, \bfh)$ is ergodic for the (translation) action of the earthquake flow on $\TT_\gamma(\bfl, \bfh)$.
Again, see \S\ref{sec: hyp twist tori} for the precise constructions.

Say that a pair of pants has property $(\nabla)$\label{a+b=c} if one of its boundary lengths is the sum of the other two.  Pick $(X,\bfh\gamma) \in \cP\TT_\gamma(\bfl, \bfh)$, and let  $\nabla_\gamma(\bfl)$ denote the number of components of $X\setminus \gamma$ that have $(\nabla)$; this number does not depend on the choice of $X$ or $\bfh$.

\begin{theorem}\label{mainthm}
Fix any pants decomposition $\gamma$, any $\bfl \in \RR^{3g-3}_{>0}$, and any $\bfh \in \RR^{3g-3}_{>0}$ such that $\bfl \cdot \bfh = 1$.
There is an earthquake flow invariant probability measure $\mu_\infty$ on $\PoM_g$ and a set $Z \subset \RR$ of zero density such that 
\[ \lim_{\substack{t \to \infty \\ t \notin Z}}
\mu_\gamma(e^t \bfl, e^{-t} \bfh) = \mu_\infty.\]
The measure $\mu_\infty$ arises from the Masur--Smillie--Veech measure on a component of a stratum of quadratic differentials having $4g-4 - 2\nabla_\gamma(\bfl)$ simple zeros and  $\nabla_\gamma(\bfl)$ zeros of order $2$. In particular,  
\begin{enumerate}
\item if $\nabla_\gamma(\bfl) = 0$, then $\mu_\infty$ is $\mu_{\Mirz}/b_g$.
\item if $\nabla_\gamma(\bfl)>0$, then $\mu_\infty$ is not equal (hence singular) to $\mu_{\Mirz}/b_g$.
\end{enumerate}
Consequently, Conjecture \ref{conj:twist tori} is true outside a set of times $e^Z$ where $Z$ has zero density.
\end{theorem}

We emphasize that Mirzakhani's conjecture only references $\mu_{\Mirz}$; it seems to have been unexpected that families of twist tori could equidistribute to anything else.
See Theorem \ref{thm: mainthm} for a detailed, precise formulation of the limiting distributions that appear in Theorem \ref{mainthm}.
The generic pair $(X, \lambda)$ with respect to one of these measures is a hyperbolic surface equipped with a measured lamination that cuts it into a union of regular ideal polygons \cite[Corollary 1.2]{shshI}.

The following averaged version of Conjecture \ref{conj:twist tori} is an immediate consequence (just push down to $\M_g$):
\[\frac{1}{T} \int_0^T \frac{\tau_{\gamma}(e^t)}{e^{(3g-3)t}}~dt \to \frac{B(X )}{b_g} d \vol_{\WP}(X).\]

The strategy of our proof is to exploit a deep connection between the earthquake and horocycle flows discovered by Mirzakhani \cite{MirzEQ} and further developed by the authors in \cite{shshI, shshII}.
Briefly, there is a bijective Borel isomorphism $\cO: \PoM_g \to \QoM_g$ that takes the earthquake flow to the horocycle flow on the moduli space $\QoM_g$ of unit area quadratic differentials in a time-preserving way.
While $\cO$ is not continuous, it is continuous along the leaves of a certain foliation and each $\TT_\gamma(\bfl,\bfh)$ is contained in such a leaf.
This allows us to map the twist torus via $\cO$ to get a torus in $\QoM_g$ consisting of flat surfaces glued together from horizontal cylinders corresponding to the components of $\gamma$ with heights $\bfh$ and lengths $\bfl$; see Figure \ref{fig:Otorus}. 
Once in this setting, we can use powerful results from Teichm{\"u}ller dynamics to understand the limiting distribution of expanding twist tori in $\QoM_g$.

\begin{figure}[ht]
    \centering
    \includegraphics[scale=.5]{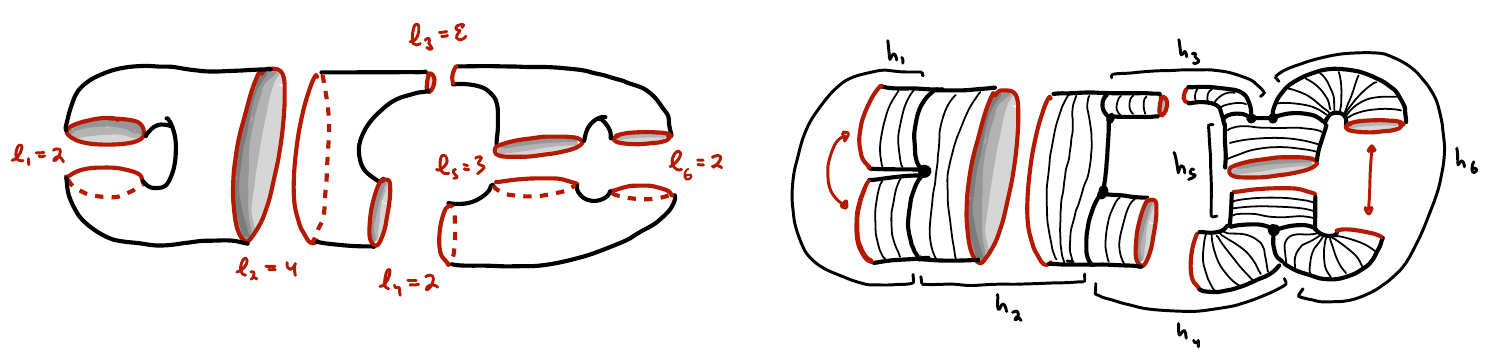}
    \caption{Hyperbolic and flat twist tori. These tori are mapped to each other under $\cO$.}
    \label{fig:Otorus}
\end{figure}

The driving engine of our proof is then the following theorem, which allows us to pull back equidistribution statements from $\QoM_g$ to $\PoM_g$.
This is a special case of \cite[Theorem B]{shshII}.

\begin{theorem}\label{thm: pull back equidistribution}
Suppose that $\nu_n \to \nu$ is a convergent sequence of probability measures on a stratum of $\QoM_g$, and suppose that $\nu$ is a $\PSL_2\RR$--invariant probability measure.
Then $\cO^* \nu_n \to \cO^* \nu$ on $\PoM_g$.
\end{theorem}

Thus, having proved an equidistribution statement for expanding twist tori in $\QoM_g$, we can simply pull it back to deduce Theorem \ref{mainthm}.
This scheme also works to show the equidistribution of many different families of expanding twist tori where $\gamma$ is an arbitrary multicurve, and can be used to show equidistribution to many different limiting distributions.
See Section \ref{sec:identifyother}, in particular Theorems \ref{thm: plumbingAD}, \ref{thm: plumbingQD},  and \ref{thm:trivalent2g}, the last of which shows that $2g$-dimensional expanding twist tori often equidistribute to Mirzakhani measure.

\begin{remark}
One could also ask about expanding twist tori in $\M_{g,n}$, that is, when the hyperbolic surfaces are allowed to have cusps.
Our methods rely on all of our previous results on the conjugacy $\cO$, which do not yet have analogues in the punctured setting.
\end{remark}

\subsection{Context: unipotent flows on moduli spaces}

There are surprising similarities between the dynamics of groups generated by unipotents on homogeneous spaces and the action of $\PSL_2\RR$ on strata of holomorphic  differentials. 
By way of \cite{MirzEQ} and \cite{shshI,shshII}, parts of this analogy can now be extended to the action (of the group $P$ of upper triangular matrices in $\PSL_2\RR$) by stretchquakes on $\PoM_g$.

A subgroup $U$ of a connected reductive Lie group $G$ is \textbf{horospherical} if there is a $g\in G$ such that 
\[U = \{u \in G : g^n u g^{-n} \to id, ~ n \to \infty\}.\]
For a lattice $\Gamma \le G$ and horospherical subgroup $U\le G$, the $U$-invariant ergodic probability measures and $U$-orbit closures in $\Gamma \backslash G$ were classified by Dani \cite{Dani:orbits}: they are homogeneous.   
In her landmark results, Ratner proved that such rigidity phenomena hold for arbitrary connected subgroups $H\le G$ generated by unipotents \cite{Ratner:measure,Ratner:top}: every $H$-invariant ergodic probability measure and $H$-orbit closure in $\Gamma \backslash G$ is homogeneous.

Despite the seemingly insurmountable difficulty that strata of holomorphic quadratic differentials on higher genus surfaces are completely inhomogeneous, Eskin and Mirzakhani showed that $P$-invariant probability measures are extremely well behaved \cite{EM}, and with Mohammadi they proved that the same is true for orbit closures \cite{EMM} (see \S\ref{sec: AISs} and Theorem \ref{thm:EMM} for a summary of their results).
\medskip

Leveraging the ergodicity of the earthquake flow, Mirzakhani showed that the analogues of expanding horospheres ($L$-level sets of the hyperbolic length of a simple closed curve) equidistribute to $\pi_*\mu_{\Mirz}$ on $\M_g$ as $L \rightarrow \infty$ \cite{Mirz_horo}.
These level sets can be lifted to $\PoM_g$ in a natural way by keeping track of the curve, as in \S\ref{subsec: hyp tt}.
There are finite, earthquake flow--invariant measures on these lifted level sets that are uniformly expanded/contracted 
in forward/backward time by the ``dilation flow,'' the diagonal part of the stretchquake action (see the end of \S\ref{subsec:conj}).
The lifted level sets have the maximum dimension with this property; see also Theorems \ref{thm:ortho homeo} and \ref{thm:conjugacy} that explain that the conjugacy $\cO$ maps lifted level sets to leaves of the unstable foliation for the Teichm\"uller geodesic flow.
These observations justify Mirzakhani's use of the ``horospherical'' terminology in \cite{Mirz_horo} (compare \cite[\S13.3]{Wright_Mirz}).

Twist tori are only $(3g-3)$-dimensional inside of the $(6g-7)$-dimensional expanding horospheres of the previous paragraph; as such, they should not be considered horospherical.  
The much more refined results of Eskin--Mirzkhani--Mohammadi, analogous to Ratner's Theorems, are vital ingredients in our proof of Theorem \ref{mainthm} and the even more delicate examples presented in \S\ref{sec:identifyother}.
Compare Figure \ref{fig:horo}.
\medskip

Mirzakhani's equidistribution results for expanding horospheres were generalized by both Arana-Herrera \cite{AH_horo} (who clarified a technical point in \cite{Mirz_horo}) and Liu \cite{Liu_horo} to different types of length functions associated to multicurves and by Arana-Herrera and the first author \cite[Theorem 5.2]{spine} to the case where one puts restrictions on the geometry of the complementary subsurface.

Our techniques recover and generalize these equidistribution results.
Indeed, let $\gamma = \gamma_1 \cup \ldots \cup \gamma_k$ be any multicurve and $\bfh \in \RR^{k}_{>0}$ any weight vector.
Consider the $L$-level set of the hyperbolic length function for $\bfh \gamma := h_1 \gamma_1 \cup \ldots \cup h_k \gamma_k$ in the Teichm{\"u}ller space $\T_g$.
A measure on this level set is {\bf horospherical} if it is absolutely continuous with respect to the Lebesgue class with continuous density, and if the restriction to each earthquake flow line is Lebesgue measure. 
These level sets can also be stratified by the combinatorial type of the ribbon graph spine for $X \setminus \gamma$ (see \S\ref{sec:background} and \S\S\ref{subsec: arcs and rgs}--\ref{sec:horospheres}).
Fixing a combinatorial type defines a submanifold of the $L$-level set, and any measure which is both absolutely continuous with respect to Lebesgue on such a submanifold and Lebesgue on earthquake flow orbits is called a {\bf horospherical slice measure}.
See Section \ref{sec:horospheres} and Figure \ref{fig:horo}.

Both horospherical and horospherical slice measures admit earthquake-flow invariant lifts to the unit cotangent bundle $\PoT_g$ by setting the second coordinate equal to $\bfh \gamma/L$.
These can further be pushed forward to $\PoM_g$, and we continue to use the same terminology.

\begin{figure}
    \centering
    \includegraphics[scale=.6]{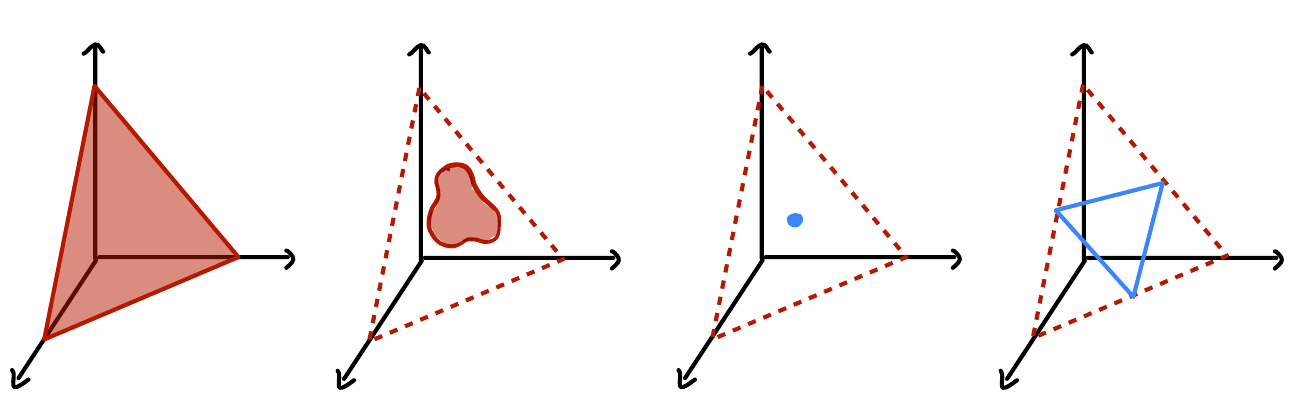}
    \caption{Projections to the length part of Fenchel--Nielsen coordinates. From left to right: the horospheres considered in \cite{Mirz_horo}, the horospheres considered in \cite{spine}, the twist tori of Conjecture \ref{conj:twist tori}, and a horospherical slice measure on a lower-dimensional stratum.}
    \label{fig:horo}
\end{figure}

As the dilation flow $\left( \Dil_t\right)$ normalizes the earthquake flow, pushforwards of horospherical (slice) measures by the dilation flow are themselves horospherical (slice) measures.
Applying a result of Forni \cite[Theorem 1.6]{Forni} (see also \cite[Proposition 3.2]{EMM:effscc}), we arrive at the following:

\begin{theorem}\label{thm: horospheres equi}
Let $\sigma$ be any horospherical (slice) measure on $\PoM_g$. Then 
\[\lim_{t \to \infty} \left( \Dil_t\right)_*\sigma
 = \mu_{\infty}\]
where $\mu_{\infty}$ is an earthquake-flow invariant ergodic probability measure on $\PoM_g$. Moreover,
\begin{enumerate}
    \item If $\sigma$ is horospherical or supported on a top-dimensional stratum, then $\mu_\infty = \mu_{\Mirz}/b_g$.
    \item If $\sigma$ is supported on a lower-dimensional stratum, $\mu_{\infty}$ is not equal (hence singular) to $\mu_{\Mirz}$.
\end{enumerate}
\end{theorem}

In the second case, the limiting measure arises via the Masur--Smillie--Veech measure on the component of the stratum of $\QoM_g$ corresponding to the horospherical slice; see \S\ref{subsec: individual pushforwards} for the proof.

\subsection*{Acknowledgments}
We are very grateful to Francisco Arana-Herrera for an inspiring discussion at the beginning of this project and to Paul Apisa for explaining his work with Alex Wright to us.
We heartily thank the referee for supplying a more direct proof of the identification of orbit closures in the setting of pants decompositions.
AC was supported by NSF grants
DMS-2005328 
and DMS-2202703. 
JF was supported by
NSF grant DMS-2005328 
and DFG grants 
427903332 
and 281071066 – TRR 191. 

\section{Flows on moduli spaces}\label{sec:background}

In this section, we discuss basic dynamical features of the spaces $\PoM_g$ and $\QoM_g$ and describe a Borel measurable bijection $\cO$ conjugating their dynamics.
Throughout, $S$ is a closed, oriented surface of genus $g\ge 2$ and $\T_g$ is the Teichm\"uller space of marked hyperbolic or conformal structures on $S$.  The mapping class group of $S$ will be denoted $\Mod_g$, and $\M_g= \T_g/\Mod_g$ is the moduli space.

\subsection{The horocycle flow}\label{subsec:horocycle flow}
Given a Riemann surface $X \in \T_g$, complex deformation theory identifies the cotangent space $T_X^* \T_g$ with the vector space of holomorphic quadratic differentials $Q(X)$ on $X$.
Let $\QT_g$ denote the bundle of non-zero quadratic differentials over $\T_g$ and $\QM_g = \QT_g/\Mod_g$.
By integrating $\pm \sqrt q$, each quadratic differential $q$ defines a flat metric on $q\setminus Z(q)$, where $Z(q)$ is the set of zeros of $q$.
The area is given by $\int |q|$, and we denote by $\QoT_g$ and $\QoM_g$ the unit area loci, which can be identified with the unit cotangent bundles of $\T_g$ and $\M_g$, respectively.

Both $\QT_g$ and $\QM_g$ and their unit area loci are stratified by the number and orders of zeros of the differentials. 
Every partition $\sing = (\kappa_1, ..., \kappa_n)$ of $4g-4$ and $\varepsilon \in \{\pm1\}$ determines a (possibly empty) \textbf{stratum} $\QT_g(\sing;\varepsilon)$ of quadratic differentials that are ($\varepsilon = +1$) or are not ($\varepsilon = -1$) squares of abelian differentials. Equivalently, the completion of the flat metric defined by $q \in \QT_g(\sing;\varepsilon)$ has $n$ cone points, each of which has cone angle excess $\kappa_i \pi$. The rotational holonomy of the metric lives in $\ZZ/2\ZZ$ and is trivial if and only if $\varepsilon = +1$.
We use $\QM_g(\sing; \varepsilon)$ to denote a stratum of $\QM_g$ and add a $1$ superscript to denote unit area loci.
Strata are not necessarily connected, but at least at the level of moduli space their connected components are classified \cite{KZstrata, Lanneau, CMexceptional}.

Strata are complex orbifolds of dimension $(2g-2) + n +\delta$, where $\delta = 1$ if $\varepsilon =+1$ and $\delta= 0$ otherwise. This can be seen via local {\bf period coordinates}, constructed as follows.
Every $q$ has a a holonomy double cover $\widehat q \to q$, branched over the zeros of odd order, so that $\widehat q$ is a (possibly disconnected) abelian differential with a holomorphic covering involution.
Let $H^1(\widehat S, \widehat Z; \CC)^-$ denote the $(-1)$-subspace for this involution, where $(\widehat S,\widehat Z)\cong (\widehat q, Z(\widehat q))$.
Take a triangulation of $q$ by {\bf saddle connections}, singularity-free geodesics joining the cone points of $q$; this can be lifted to a triangulation $\widehat {\mathsf T}$ of $\widehat q$ by saddle connections. 
Integrating $\widehat q$ over the edges of $\widehat {\mathsf T}$ defines a cohomology class in $H^1(\widehat S, \widehat Z; \CC)^-$, and any small deformation of all of the triangles subject to the constraints imposed by $H^1(\widehat S, \widehat Z; \CC)^-$ defines a new quadratic differential in the same stratum.
The period vector mapping into $H^1(\widehat S,\widehat Z;\CC)^-$ near $q$ is called a \textbf{period coordinate chart}.
Two overlapping period coordinate charts differ by an integer linear transformation. We refer the reader to \cite{Zsurvey}, \cite{Wright_Mirz}, and \cite{AMbook} for a more thorough discussion of period coordinates and background on flat surfaces.

In particular, for each $q \in \QT_g(\sing; \varepsilon)$, the tangent space may be locally modeled by
\[T_q\QT_g(\sing;\varepsilon) \cong H^1(\widehat S,\widehat Z;\CC)^-.\]
The same is true for $\QM_g(\sing; \varepsilon)$ modulo orbifold issues (which can be resolved by passing to a finite cover).

There is an action of $\SL_2\RR$ on $\QM_g$  preserving each stratum component $\cQ$ as well as its unit-area locus $\cQ^1$.
For $M \in \SL_2\RR$ and $q\in \cQ$, $M \cdot q\in \cQ$ is defined by postcomposing charts for the singular flat metric obtained by integration 
with the $\RR$-linear map $M: \CC \to \CC$.
Since these charts are only defined up to $\pm 1$, this $\SL_2\RR$ action factors through $\PSL_2 \RR$.
In a period coordinate chart $H^1(\widehat {S}, \widehat {Z};\CC)^- \cong \CC \otimes H^1(\widehat {S}, \widehat {Z};\RR)^-$, 
the action is given by applying $M$ to $\CC$.
The one-parameter subgroups 
\[A = \left\{g_t = \begin{pmatrix}
    e^{t} & \\
     & e^{-t}
\end{pmatrix} \mid t\in \RR\right\}
\hspace{1em} \text{ and } \hspace{1em} U =  
\left\{u_s = \begin{pmatrix}
    1 & s\\
     &1
\end{pmatrix} \mid t\in \RR\right\}\]
induce the \textbf{Teichm\"uller geodesic flow}
and \textbf{unstable Teichm\"uller horocycle flow}, respectively 
We denote by $P\le \PSL_2\RR$ the group of upper-triangular matrices generated by geodesic and horocycle flows.

Each component $\cQ$ of each stratum of quadratic differentials has a natural affine measure coming from period coordinate charts.
This induces a $\PSL_2\RR$-invariant ergodic {\bf Masur--Smillie--Veech probability measure} $\nu_\cQ^1$ on the unit area locus $\cQ^1$
by assigning to a set $E \subset \cQ^1$ the affine measure of the cone $\{tE \mid t \in (0,1]\}$, divided by the total mass of $\cQ^1$.
A characterization of the $\PSL_2\RR$-invariant ergodic probability measures on $\QoM_g$ was given by Eskin--Mirzakhani \cite{EM}.  See \S\ref{sec: AISs} and Theorem \ref{thm:EMM} below.

\subsubsection*{Differentials and foliations}
Every $q\in \QM_g$ is equipped with an \textbf{imaginary foliation} $\Im(q)$ by horizontal lines (which supports a  transverse measure recording the total variation of the imaginary part of $q$) as well as a \textbf{real foliation} $\Re(q)$ by vertical lines supporting a transverse measure locally of the form $|dx|$.
Let $\MF_g$ denote Thurston's space of equivalence classes of measured foliations (where two foliations are equivalent if they differ by Whitehead moves and isotopy). 
An important result of Gardiner--Masur building on work of Hubbard--Masur shows that the map recording the equivalence classes of $\Im(q)$ and $\Re(q)$
\begin{equation}\label{eqn:GM}
    \QT_g \to \MF_g \times \MF_g \setminus \Delta
\end{equation}
is a homeomorphism \cite{GM,HubMas}, where
\[\Delta := \{(\eta, \lambda) \mid \exists \gamma \in \MF_g \text{ such that } i(\gamma, \eta) + i(\gamma, \lambda) = 0\}.\]

Following \cite{Veechgeo} and \cite{Fornidev} (see also \cite[Theorem 3.15]{ABEM}), we define a pair of transverse foliations of $\QT_g$ and $\QoT_g$.
Given $q \in \QM_g$, define 
\[\mathcal W^{uu}(\lambda) := \{q \in \QT_g: \text{ the imaginary foliation } \Im(q) \text{ is equivalent to } \lambda\}.\]
The intersection of $\mathcal W^{uu}(\lambda)$ with the unit area locus is often referred to as the {\bf strong unstable} foliation for the Teichm\"uller geodesic flow. The foliation $\mathcal W^{ss}(\lambda)$ is defined similarly with the imaginary part replaced with the real, and its intersection with $\QoT_g$ is often referred to as the {\bf strong stable} foliation.
Each leaf $\mathcal W^{uu}(\lambda)$ has a stratified real-analytic structure coming from restricting the linear structure of strata of $\QT_g$ to intersections with $\mathcal W^{uu}(\lambda)$.

For every $\lambda$, define $\MF(\lambda):= \{ \eta \in \MF_g \mid (\eta, \lambda) \notin \Delta\}.$ 
Using \eqref{eqn:GM}, we have
\[\MF(\lambda) \cong \mathcal W^{uu}(\lambda) \subset \QT_g,\]
which we can use to equip $\MF(\lambda)$ with a stratified real-analytic structure.

\subsection{The earthquake flow}
By the uniformization theorem, $\T_g$ can also be identified with the space of marked hyperbolic structures on $S$.
This viewpoint yields global analytic \textbf{Fenchel--Nielsen} coordinates on $\T_g$ adapted to any pair of pants decomposition $\gamma$, comprised of $3g-3$ coordinate length functions $\ell_i$ and $3g-3$ real valued twist functions $\tau_i$.
A choice of Fenchel--Nielsen coordinates requires a choice of what ``zero twist'' should mean, however, the differentials $d\tau_i$ are defined. In other words, the twisting data form a principal $\RR^{3g-3}$-bundle over the length data.

Wolpert proved that the symplectic $2$-form 
\[\omega_{\WP} = \sum_{i =1}^{3g-3} d\ell_i \wedge d\tau_i\]
coincides with the imaginary part of the Weil--Petersson K\"ahler form, hence is invariant by $\Mod_g$ \cite{WolWP}.
The top wedge power of $\omega_{\WP}$ defines a volume form and corresponding {\bf Weil--Petersson measure} $\vol_{\WP}$ on $\M_g$. This measure is in the class of Lebesgue and has finite total mass.

For any $X \in \T_g$, Thurston introduced the space $\ML(X)$ of \textbf{measured geodesic laminations} as a natural completion of the set of simple closed geodesics on $X$ \cite{Thurston:bulletin,CB,PennerHarer}.
A \textbf{geodesic lamination} $\lambda \subset X$ is a closed set that is foliated by disjoint, simple, complete geodesics.
The metric completion $X\setminus \lambda$ of its complement is a finite area hyperbolic surface with (possibly non-compact) totally geodesic boundary.
A \textbf{transverse measure} on $\lambda$ is 
the assignment of a measure $\mu_k$ to each $C^1$ arc $k$ transverse to $\lambda$, such that the support of $\mu_k$ is $k\cap \lambda$, and such that this assignment is natural under inclusion and invariant under transverse isotopy.
Throughout, we will write $\lambda$ to refer both to a transverse measure and its supporting geodesic lamination.

Neither the set of geodesic laminations nor the space $\ML(X)$ depends on the metric $X$; as such, we use the notation $\ML_g$ to refer to the space of measured laminations.
The theory of train tracks gives a family of coordinate charts specifying polyhedral cones in $\ML_g$; in each chart, $\mathbb N$-weighted multicurves specify an integer lattice. Transition maps between these charts are in $\SL_n$, and hence this lattice can be used to define a $\Mod_g$-invariant Lebesgue-class measure on $\ML_g$ called the \textbf{Thurston measure} $\mu_{\Th}$.
The generic $\lambda$ with respect to $\mu_{\Th}$ is {\bf maximal}, in that $X \setminus \lambda$ is a union of ideal triangles.

Given $X \in \T_g$, the function that sends a multicurve $\gamma$ to its geodesic length on $X$ extends homogeneously and continuously to the space of measured laminations.
Denote by $\ML^1(X)$ the sphere of unit length measured laminations on $X$;
then $\mu_{\Th}$ defines a finite measure $\hat{\mu}_{\Th}^X$ on $\ML^1(X)$ by
\[\hat{\mu}_{\Th}^X(E) := 
\mu_{\Th}\{ c \lambda \mid \lambda \in E, c \in [0,1]\}.\]
Let $B(X)$ denote the mass of $\hat{\mu}_{\Th}^X$, equivalently, the $\mu_{\Th}$ measure of the ball of length $\le 1$ laminations on $X$.

There is a flat bundle $\PT_g$ over $\T_g$ with fiber $\ML(X)$ over $X$; the quotient by $\Mod_g$ is denoted by $\PM_g$.
Similarly to quadratic differentials, we will decorate with a superscript $1$ when we restrict to the locus of unit-length measured laminations.
By \cite[Theorem 5.1]{Th_stretch}, these bundles can also be identified with the (unit) cotangent bundles of $\T_g$ and $\M_g$, respectively.
Integrating the fiberwise measures $\hat{\mu}_{\Th}^X$ against the Weil--Petersson volume yields the $\Mod_g$-invariant, Lebesgue-class \textbf{Mirzakhani measure} on $\PoT_g$.
Its local pushforward to $\PoM_g$ is a finite measure $\mu_{\Mirz}$ that we also call by the same name. The total total mass of $\mu_{\Mirz}$ is
\[b_g := \int_{\M_g} B(X) \, d\vol_{\WP}(X),\]
and by construction, its pushforward to $\M_g$ under the forgetful map is $B(X) d \vol_{\WP}$(X).

The Mirzakhani measure is invariant and ergodic for the (right) {\bf earthquake flow}, defined as follows.
Given a $(X,c\gamma) \in \PoM_g$, where $c\gamma$ is a weighted simple closed curve, $\Eq_s(X,c\gamma) = (\Eq_s(X,c\gamma), c\gamma)$, where the first coordinate is the hyperbolic metric obtained by cutting $X$ open along $\gamma$ and re-gluing with a right twist of length $cs$.
Equivalently, this flow translates by $cs$ in Fenchel--Nielsen coordinates adapted to a pants decomposition containing $\gamma$.
For $\lambda \in \ML(S)$ approximated by weighted simple closed curves $c_n \gamma_n$, the homeomorphisms 
$X\mapsto \Eq_s(X,c_n\gamma_n)$ converge to a map called the right earthquake in $\lambda$ (see \cite{Kerckhoff_NR} or \cite{Th:EQ}).
This map also admits an independent description by cutting and re-gluing all of the complementary components of $\lambda$ on $X$ with a rightwards shear, moving all components relative to one another by amounts dictated by the transverse measure to $\lambda$.

\subsection{The conjugacy}\label{subsec:conj}
There is an important connection between the horocycle and earthquake flows introduced above: they are Borel isomorphic (see Theorem \ref{thm:conjugacy} below).
This relationship was first observed by Mirzakhani on the full $\mu_{\Mirz}$-measure subset of pairs $(X,\lambda) \in \PoM_g$ where $\lambda$ is a maximal geodesic lamination \cite{MirzEQ} and was extended to the entirety of $\PoM_g$ by the authors \cite{shshI}.

This isomorphism is constructed by assigning a ``dual'' foliation to every lamination $\lambda$ on a hyperbolic surface $X$.
Consider a component $Y$ of $X \setminus \lambda$.
Most points of $Y$ have a unique closest point on $\partial Y$; those that have multiple closest points on $\partial Y$ form a piecewise geodesic $1$-complex $\Sp(Y)$ called the \textbf{spine} of $Y$ (see also \S\ref{subsec: arcs and rgs} below).
The fibers of the closest point--projection map $Y\setminus \Sp(Y) \to \partial Y$ form a foliation of $Y\setminus \Sp(Y)$. Extending continuously across $\Sp(Y)$ yields a piecewise geodesic singular foliation $\cO_{\partial Y}(Y)$ called the \textbf{orthogeodesic foliation}.
This foliation has $k$-pronged singularities at points that are equidistant from $k$ different points of $\partial Y$, and every endpoint of every leaf of $\cO_{\partial Y}(Y)$ meets $\partial Y$ orthogonally.

Because their leaves are always orthogonal to $\lambda$ and the line field tangent to $\lambda$ is Lipschitz continuous, the orthogeodesic foliations on the components of $X\setminus \lambda$ can be extended continuously across $\lambda$, yielding a singular foliation $\cO_{\lambda}(X)$ on $S$.
This foliation also carries a natural transverse measure, defined by assigning to any geodesic arc in $\lambda$ its Lebesgue measure and then extending uniquely to all arcs via transverse isotopy.
See \cite[Section 5]{shshI} for more details.

This defines a map $\cO_\lambda: \T_g \to \MF_g$ recording the equivalence class of $\cO_\lambda(X)$; the following theorem records the global properties of $\Ol$.
When $\lambda$ is maximal, it is due to Thurston \cite{Th_stretch}; see also \cite{Bon_SPB} for a treatment in terms of transverse H\"older distributions.
For arbitrary $\lambda$, the following is a direct consequence of Theorems D and E of \cite{shshI}.

\begin{theorem}\label{thm:ortho homeo}
For every $\lambda\in \ML_g$, the map $\Ol$ is a stratified real-analytic homeomorphism
\[\Ol:\T_g \to \MF(\lambda) \cong \mathcal W^{uu}(\lambda)\]
that is equivariant with respect to the stabilizer of $\lambda$ in $\Mod_g$ and satisfies $i(\cO_\lambda(X), \lambda) = \ell_X(\lambda)$.
\end{theorem}

\begin{remark}
    Theorem D of \cite{shshI} establishes that $\Ol$ is a homeomorphism. Theorem E therein asserts that $\Ol$ factors as $\Ol = I_\lambda\inverse\circ \sigl$, where $\sigl$ is stratified real-analytic and $I_\lambda\inverse$ is linear when restricted to each stratum. Thus $\Ol$ is also stratified real-analytic.
\end{remark}

It is well known that $\ML_g$ and $\MF_g$ are canonically homeomorphic \cite{Levitt}.
Combining this fact with Theorem \ref{thm:ortho homeo} and \eqref{eqn:GM}, 
we obtain a $\Mod_g$-equivariant bijection
\[\begin{array}{rcccl}
\cO: \PT_g = & \T_g \times \ML_g & \to & \MF_g \times \MF_g \setminus \Delta & \cong \QT_g. \\
    & (X, \lambda) & \mapsto & (\Ol(X), \lambda)
\end{array}\]
Moreover, for $(X,\lambda) \in \PT_g$, the real foliation of $\cO(X,\lambda)$ is actually {\em isotopic} to $\cO_\lambda(X)$ (see \cite[\S5.3]{shshI}).

The fact that the length of $\lambda$ on $X$ is equal to the area of $\cO(X, \lambda)$ combined with the $\Mod_g$-equivariance of $\cO$ imply that it restricts to a bijection 
\[\cO:\PoM_g \to \QoM_g\]
of the same name.
The following omnibus theorem combines the main theorems of \cite{MirzEQ} and \cite{shshI} together with a special case of the main result of \cite{shshII}. See the footnotes for specific citations.

\begin{theorem}\label{thm:conjugacy}
The map $\cO:\PoM_g \to \QoM_g$ is a bijective Borel isomorphism\footnote{Theorem C of \cite{shshI} and Theorem 15.4 of \cite{shshII}.} with the following properties:
\begin{enumerate}
    \item $\cO$ takes earthquake flow to horocycle flow in a time-preserving way: $\cO \circ \Eq_s = u_s \circ \cO$.\footnote{Theorem A of \cite{shshI}; see also Theorem 1.1 of \cite{MirzEQ}.}
    \item $\cO_* \mu_{\Mirz}/ b_g$ is the Masur--Smillie--Veech probability measure on the principal stratum of $\QoM_g$.\footnote{Section 1.2 of \cite{MirzEQ}.}
    \item Each $k$-pronged singularity of $\Ol(X)$ corresponds to an order $(k-2)$ zero of $\cO(X, \lambda)$ and vice versa.\footnote{Proposition 5.10 of \cite{shshI}.}
    \item If $\cQ^1$ is a component of a stratum of $\QoM_g$ and $q \in \cQ^1$ has no horizontal saddle connections, then $q$ is a point of continuity for $\cO\inverse|_{\cQ^1}$.\footnote{A stronger version of this is proved as Theorem 15.6 of \cite{shshII}.}
\end{enumerate}
\end{theorem}
All $\PSL_2\RR$-invariant probability measures on strata give zero measure to the set of differentials with horizontal saddle connections (see \S\ref{sec: AISs} below).
Item (4) therefore allows us to pull back equidistribution to such measures \cite[Theorem 5.1]{B:measures}, and implies Theorem \ref{thm: pull back equidistribution} from the Introduction.

\begin{remark}
Mirzakhani pointed out that $\cO$ (or any extension of her original conjugacy) is discontinuous along sequences $(X_n,\lambda_n)$ where the measured and Hausdorff limits of maximal laminations $\lambda_n$ disagree \cite[p. 33]{MirzEQ}.
It is now known that there can be no continuous map $\PoM_g \to \QoM_g$ conjugating the earthquake and horocycle flows \cite{AHW:continuous}, so Theorem \ref{thm:conjugacy} (and the more general Theorem A of \cite{shshII}) are essentially optimal.
\end{remark}

The conjugacy $\cO$ allows us to pull back the Teichm{\"u}ller geodesic flow from $\QoM_g$ to define a measurable {\bf dilation flow} $(\Dil_t)$ on $\PoM_g$ that normalizes the earthquake flow (see \cite[\S 15.3]{shshI}):
\[\Dil_t \circ \Eq_s \circ \Dil_{-t} = \Eq_{e^{2t}s}.\]
This transformation can also be characterized by scaling the orthogeodesic foliation (and hence the spine of $X \setminus \lambda$).
Given any $X \in \T_g$ and $\lambda \in \ML_g$, the {\bf dilation ray $\{X_t\}$ directed by $\lambda$} is the projection of $\Dil_t(X, \lambda)$ to $\T_g$.
By Theorem \ref{thm:ortho homeo}, it is uniquely specified by the equation 
$\Ol (X_t) = e^t \Ol(X).$
 
\section{Hyperbolic twist tori}\label{sec: hyp twist tori}
We begin by establishing some notation for twist tori and associated measures on moduli spaces.
Throughout this section, we use Conjecture \ref{conj:twist tori} as a motivating example but phrase all of our results in generality.

\subsection{Twist tori}\label{subsec: hyp tt}
Let $\gamma = (\gamma_1, \ldots, \gamma_n)$ be a multicurve on $S$; throughout the paper we use $S \setminus \gamma$ to denote the (possibly disconnected) compact surface with boundary obtained by removing an open annulus about each $\gamma_i$.
Associated to $\gamma$ there is a {\bf cut-and-glue fibration} 
that expresses a hyperbolic structure on $S$ in terms of the induced hyperbolic structure on $S \setminus \gamma$ and how these structures are glued together.
Let $\T(S \setminus \gamma)$ denote the space of marked\footnote{Here, the marking does not record twisting about the boundary.}
hyperbolic structures with totally geodesic boundary on $S \setminus \gamma$ so that for each $i$, the two boundaries of $S \setminus \gamma$ that are glued along $\gamma_i$ have the same length.
Then we have the following principal $\RR^n$-bundle:
\[\begin{tikzcd}
\RR^n \arrow{r}
& \T_g \arrow{d}{\text{cut}_\gamma}\\
& \T(S\setminus \gamma)
\end{tikzcd}\]
where the $\RR^n$ fibers record the twisting about the curves of $\gamma$.
In the case that $\gamma$ is a pair of pants, these are just the usual Fenchel--Nielsen coordinates. For an arbitrary measured lamination $\lambda$, the authors developed {\em shear-shape coordinates} that parametrize $\T_g$ in a similar way \cite{shshI}.

Fix any $Y \in \T(S \setminus \gamma)$ and set $\bfl = \bfl(Y)$ to be the vector of boundary lengths of $Y$ (that is, the lengths of the curves of $\gamma$). The fiber cut$_\gamma^{-1}(Y)$ is an $\RR^n$-torsor, hence is equipped with $n$-dimensional Lebesgue measure.
The twist subgroup 
$\langle T_{\gamma_i}\rangle \cong \ZZ^n < \Mod(S)$
acts by translations on the fiber; quotienting yields an $n$-torus 
\[\text{cut}_{\gamma}^{-1}(Y) / \langle T_{\gamma_i}\rangle
\cong (\RR / \ell_1 \ZZ) \times \ldots \times (\RR / \ell_n \ZZ)
\subset \T_g / \langle T_{\gamma_i}\rangle\]
equipped with a measure $\widetilde{\tau}_\gamma(Y)$ of mass $\prod \ell_i$(which is the local pushforward of Lebesgue).

Let $\pi: \T_g \to \M_g$ and $\tilde{\pi}: \T_g / \langle T_{\gamma_i}\rangle \to \M_g$ denote the natural (orbifold) covering maps.

\begin{definition}\label{def:twisttori}
The {\bf twist torus} associated to $Y \in \T(S \setminus \gamma)$ is 
\[\TT_{\gamma}(Y):= \pi\left(\text{cut}_{\gamma}^{-1}(Y) \right)
=\tilde{\pi} \left( \text{cut}_{\gamma}^{-1}(Y) / \langle T_{\gamma_i}\rangle \right) 
\subset \M_g.\]
The associated {\bf twist torus measure} on $\M_g$ is
$\tau_\gamma(Y):= \tilde{\pi}_* \left( \widetilde{\tau}_\gamma(Y) \right).$
\end{definition}


\subsubsection*{Lifting to the cotangent bundle}
In order to apply the results of \cite{shshII} and connect hyperbolic twist tori to their flat counterparts, we will need to first lift twist tori to $\PoM_g$.
To do so, choose any weight vector $\bfh = (h_1, \ldots, h_n) \in \RR_{>0}^n$ satisfying 
$\bfl \cdot \bfh = 1$.
We can then lift the fiber $\text{cut}_{\gamma}^{-1}(Y)$ into $\PoT_g$ by setting the second coordinate equal to ${\bfh}\gamma$, the weighted multicurve ${h_1} \gamma_1 \cup \ldots \cup {h_n} \gamma_n$.
Quotienting out by the twist subgroup yields an embedded torus in $\PoT_g /  \langle T_{\gamma_i}\rangle$ with its Lebesgue measure, and pushing down to moduli space we get a torus
\[\cP\TT_{\gamma}(Y, \bfh) := 
\pi \left(
\text{cut}_{\gamma}^{-1}(Y) \times \{ \bfh \gamma\}
\right) 
\subset \PoM_g\]
where $\pi$ is again the quotient by the action of the mapping class group.

We can now define a measure supported on $\cP \TT_{\gamma}(Y, \bfh)$ similarly to how we obtained $\tau_\gamma(Y)$:
there is a natural Lebesgue measure on $\text{cut}_\gamma^{-1}(Y) \times \{\bfh \gamma\}$, whose local pushforward yields a finite measure after quotienting by the twist subgroup, and we define $\mu_\gamma (Y, \bfh)$ to be the further pushforward of this measure to $\PoM_g$, {\em normalized to have unit mass}.
Unpacking definitions, we see that
\begin{equation}\label{eqn:pushtoridown}
p_* \left(\mu_\gamma (Y, \bfh) \right) = \frac{\tau_\gamma(Y)}{\ell_1 \cdots \ell_n}
\end{equation}
as measures on $\M_g$, where $p:\PoM_g \to \M_g$ is the natural projection map.

\begin{remark}
The map $p$ does not necessarily restrict to a homeomorphism between $\TT_\gamma(Y)$ and $\cP\TT_\gamma(Y, \bfh)$; it is generally just an (orbifold) covering.
For example, if $\gamma = (\gamma_1, \gamma_2)$ where each $\gamma_i$ is nonseparating but $h_1 \neq h_2$, then the stabilizer of $\gamma$ is larger than the stabilizer of $\bfh \gamma$.
\end{remark}

\subsection{Arc systems and ribbon graphs}\label{subsec: arcs and rgs}
In order to discuss families of expanding twist tori, we will need to vary the hyperbolic structure $Y$ on the complement of $\gamma$.
This is easy when $\gamma$ is a pants decomposition, because one can just scale up the lengths of the pants curves.
In general, we will consider the pushforwards of $\cP\TT_\gamma(Y, \bfh)$ by the dilation flow (defined in \S\ref{subsec:conj} above).

To describe these tori more concretely, we recall from Section \ref{subsec:conj} that the spine $\Sp(Y)$ of a hyperbolic surface with boundary $Y$ is the set of points that have multiple closest points on $\partial Y$.
The spine is naturally a ribbon graph (i.e., a graph with cyclic orderings at its vertices) whose vertices of valence $k$ correspond to $k$-pronged singularities of $\cO_{\partial Y}(Y)$.
It is also equipped with a metric given by assigning to each edge of $\Sp(Y)$ the length of either of its closest point--projections to $\partial Y$; equivalently, this is the mass of the transverse measure deposited by the orthogeodesic foliation on the geodesic connecting the two vertices.

Dually, the data of $\Sp(Y)$ can also be recorded as a weighted, filling arc system on $Y$.
Leaves of the orthogeodesic foliation $\cO_{\partial Y}(Y)$ break up into finitely many isotopy classes of arcs rel $\partial Y$ and we use $\arc$ to denote the set of these isotopy classes.
Assigning to each arc the length of its dual edge of $\Sp(Y)$ then defines a positive weight system on $\arc$.
See \cite[\S6]{shshI} for a more thorough discussion of this duality.

The map taking a taking a hyperbolic surface with boundary to its spine (or dual arc system) in fact defines a stratified piecewise real-analytic $\Mod_{g,b}$--equivariant homeomorphism \cite{Luo}
\[\Sp: \T_{g,b} \xrightarrow{\sim} \TRG_{g,b},\]
where $\T_{g,b}$ is the space of marked hyperbolic surfaces with genus $g$ and $b$ boundaries and $\TRG_{g,b}$ is the space of marked metric ribbon graphs (up to homotopy) with the same genus and number of boundaries.\footnote{A {\bf marking} of a ribbon graph $\Gamma$ is a choice of deformation retract $S_{g,b} \to \Gamma$.}
See also \cite{Ushijima, Do, Mond_handbook}.

The spine map takes surfaces with boundary lengths $\bfl$ to ribbon graphs whose boundary lengths are also $\bfl$.
Given a multicurve $\gamma$ on a closed surface $S$, we can therefore piece together the spine maps on each component of $S \setminus \gamma$ to get a homeomorphism
\begin{equation}\label{eqn:spine T(S minus gamma)}
\Sp:\T(S \setminus \gamma) \xrightarrow{\sim} \TRG(S \setminus \gamma)
\end{equation}
where $\TRG(S \setminus \gamma)$ is defined analogously to $\T(S \setminus \gamma)$.
More specifically, each curve of $\gamma$ corresponds to two different cycles in the ribbon graph spines; the lengths of these cycles must match. See \cite[pg. 39]{spine}.

We observe that when $\gamma$ is a pants decomposition, $\Sp(Y)$ is completely determined by the boundary lengths $\bfl$.
In the general case, the lengths of the $\gamma_i$ are explicitly obtainable from the metric on $\Sp(Y)$, and we will continue to use
$\bfl \in \RR^n$ to denote the vector of lengths of the curves of $\gamma$.

\begin{notation}In light of \eqref{eqn:spine T(S minus gamma)}, throughout the rest of the paper we will refer to twist tori and their measures in terms of the complementary metric ribbon graph $\rg = \Sp(Y) \in \TRG(S \setminus \gamma)$. For example,
\[\mu_\gamma(\rg):=\mu_\gamma(Y)
\text{ and }
\cP \mathbb T_\gamma(\rg, \bfh) := \cP \mathbb T_\gamma(Y, \bfh).\]
\end{notation}

The discussion in this section should make it clear how a twist torus changes when it is pushed forward by the dilation flow: the spine of the complementary subsurface is rescaled. That is,
\begin{equation}\label{eqn:dilationtori}
\Dil_t \left( \cP\TT_\gamma(\mathsf Y, \bfh) \right)
= \cP\TT_\gamma(e^t \mathsf{Y}, e^{-t}\bfh).
\end{equation}

\subsubsection*{Walls and chambers}
The space $\TRG(S \setminus \gamma)$ has a natural piecewise-linear structure and is stratified by the combinatorial type of the ribbon graphs. To avoid confusing these with strata of differentials, we will call such strata {\bf facets}.\footnote{The space of all ribbon graphs on $S \setminus \gamma$ is naturally a union of open orthants, and the gluing conditions cut out linear subspaces of each orthant.}
Top-dimensional facets correspond to trivalent metric ribbon graphs and lower-dimensional ones correspond to ribbon graphs with higher valence vertices --- this combinatorial structure is also readily apparent in the dual picture in terms of weighted filling arc systems.
For a given ribbon graph $\Gamma$ of the same topological type as $S \setminus \gamma$, we set $F(\Gamma)$ to denote the facet of $\TRG(S \setminus \gamma)$ consisting of metrics on $\Gamma$ satisfying the gluing conditions. 
Luo's work \cite{Luo} implies that $\Sp$ restricts to a real-analytic homeomorphism between each facet and a corresponding subvariety of $\T(S \setminus \gamma)$.

It is important to note that that the dimension of $F(\Gamma)$ is not necessarily the number of edges of $\Gamma$.
For the same reason, some choices of $\Gamma$ do not correspond to (nonempty) facets: this is because combinatorial conditions on the spine can enforce constraints on the lengths of the boundaries, which might not be compatible with the gluing conditions over the curves of $\gamma$.
See Figure \ref{fig:emptyface} for an example of this phenomenon.

\begin{figure}
    \centering
    \includegraphics[width=0.6\linewidth]{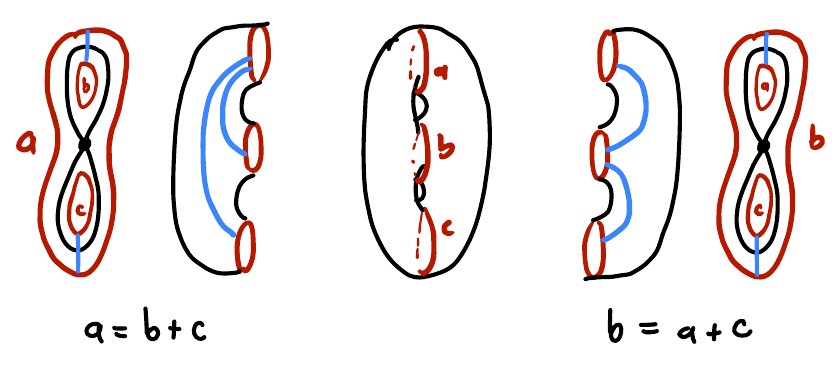}
    \caption{A pair of spines that does not appear in $\TRG(S \setminus \gamma)$ and their dual arc systems.
    Together, the combinatorial conditions imply that the length of $c$ must be 0, which cannot happen inside $\T(S \setminus \gamma)$.}
    \label{fig:emptyface}
\end{figure}

\subsection{Horospheres and slices}\label{sec:horospheres}

Given any $\bfh = (h_1, \ldots, h_n) \in \RR_{>0}^n$, define the length $1$ {\bf horosphere} in $\T_g$ based at the weighted multicurve ${\bfh}\gamma ={h_1} \gamma_1 \cup \ldots \cup {h_n} \gamma_n$ to be the $1$-level set of the (total) length function.\footnote{This discussion still holds for the level sets of other functions of the length of the components of $\gamma$; compare \cite{AH_horo,Liu_horo}.}
That is, 
\[H_{\gamma}(\bfh) := \left\{ X \in \T_g \mid \sum_{i=1}^n h_i \ell_X(\gamma_i) = 1 \right\}.\]
Each horosphere has a natural lift $\cP H_{\gamma}(\bfh)$ to $\PoT_g$ with $\bfh \gamma$ as the second coordinate.
Both the horosphere and its lift are homeomorphic to $\RR^{6g-7}$ and are foliated by (lifts of) twist tori.

Following Forni \cite{Forni}, we define a {\bf horospherical measure} to be any probability measure supported on $\cP H_{\gamma}(\bfh)$ that is absolutely continuous with respect to the Lebesgue class with continuous density and such that its restriction to each earthquake flow line is the restriction of Lebesgue to an interval.
We also say that the pushforward of such a measure to $\PoM_g$ is horospherical.\footnote{
For example, in \cite{AH_horo, Liu_horo, spine}, such measures are obtained by taking certain $\Stab(\bfh \gamma)$-invariant measures on $\cP H_{\gamma}(\bfh)$, taking a  local pushforward to the intermediate quotient $\PoT_g /\Stab(\bfh \gamma)$, then pushing down to $\PoM_g$.}

Following our discussion of the wall-and-chamber structure of $\TRG(S \setminus \gamma)$, given any facet $F(\Gamma)$, we also define the corresponding length $1$ {\bf horosphere slice} as follows:
\[H_{\gamma}(\bfh,\Gamma) := 
\left\{ X \in \T_g \mid \sum_{i=1}^n h_i \ell_X(\gamma_i) = 1
\text{ and }
\Sp(X \setminus \gamma) \in F(\Gamma)
\right\}.\]
As before, there is a natural lift to $\PoT_g$, both the horosphere slice and its lift are homeomorphic to a Euclidean space, and both are foliated by twist tori.
We say that a probability measure supported on $\cP H_{\gamma}(\bfh,\Gamma)$ is a {\bf horosphere slice measure} if it is in the class of Lebesgue with continuous density (now on the slice) and the conditional measures on earthquake flow lines are Lebesgue on intervals.
Just as above, we also use this term to denote pushforward of such a measure to $\PoM_g$.

\begin{remark}
There are natural examples of horospherical and horosphere slice measures coming from restricting the Weil-Petersson form or by integrating the natural affine measure on facets of $\TRG(S \setminus \gamma)$ against twist parameters.
For applications, one usually wishes to fix such a choice and compute its mass (pre-normalization).
For this reason, the definition of horospherical measures is usually a little more involved in order to record factors coming from orbifold issues.
We refer the interested reader to \cite[\S7]{spine} or \cite{AH_compcount}.
\end{remark}

\section{Flat twist tori}\label{sec: flat twist tori}
In this section, we discuss the flat counterparts of twist tori, horospheres, slices, and the relevant measures. 
We prove using the strong rigidity phenomena of \cite{EM,EMM} for the $\PSL_2\RR$-action on strata of quadratic differentials that (the union of) such expanding families equidistribute to the affine measure on an affine invariant subvariety.

\subsection{Twist tori and horospheres}\label{sec: flat tt and horo}
Corresponding to $\cP\TT_{\gamma}(\rg, \bfh)$ is a \textbf{flat twist torus} parameterizing quadratic differentials built from horizontal cylinders with core curves $\gamma$, heights $\mathbf h$, and lengths (circumferences) $\bfl$; these cylinders are glued with arbitrary twists along the metric ribbon graph $\rg$.
Indeed, define 
\[\cQ\mathbb T_\gamma(\rg,\bfh) = \cO(\cP\mathbb T_\gamma(\rg,\bfh)) \subset \QoM_g\]
and let
\[\nu_\gamma(\rg,\bfh) = \cO_*\mu_\gamma(\rg,\bfh)\]
be the associated probability measure.
By Theorem \ref{thm:ortho homeo}, $\cO$ restricts to a real analytic homeomorphism between  $\cP\mathbb T_\gamma(\rg, \bfh)$ and  $\cQ\mathbb T_\gamma(\rg, \bfh)$.

The differentials $q \in \cQ\mathbb T_\gamma(\rg,\bfh)$ are completely horizontally periodic, because their horizontal foliations are measure equivalent to $\bfh \gamma$, i.e., correspond to one another under the natural identification between $\MF_g$ and $\ML_g$ \cite{Levitt}.
With definitions as in \cite{Wright:cylinder} or \S\ref{subsec:cylrank} below, the flat twist torus
$\cQ\mathbb T_\gamma(\rg,\bfh)$ is just the orbit of any such by the abelian group $\RR^\gamma$ of twists in the horizontal cylinders, and $\nu_\gamma(\rg,\bfh)$ is Lebesgue on $\cQ\mathbb T_\gamma(\rg,\bfh)$.
We observe that $\cQ\mathbb T_\gamma(\rg,\bfh)$ is locally an affine submanifold in period coordinates, cut out by equations of the form $\Im(z) = h_i$ and $z=\ell(e)$, where $e$ is an edge of $\rg$ with length $\ell(e)$.
By our convention that $\bfl \cdot \bfh =1$, the intersection of these affine subspaces lies in the unit area locus.

The union of the horizontal saddle connections of any $q \in \cQ\mathbb T_\gamma(\rg,\bfh)$ is exactly $\rg$.
Thus, the angle around each singularity of the flat metric is $\pi$ times the valence of the corresponding vertex of $\rg$; see also \S2 of \cite{shshI}.
However, in order to identify the ambient stratum containing $\cQ\mathbb T_\gamma(\rg,\bfh)$ we also need to understand if $q$ is globally the square of an abelian differential. This can be checked combinatorially, as follows.

\begin{definition}\label{def: orientable}
    Let $\Sigma$ be a compact surface with non-empty boundary and let $\Gamma$ be a ribbon graph with topological type $\Sigma$. We say that $\Gamma$ is {\bf orientable} if its edges can be oriented so that each component of $\partial \Sigma$ inherits a consistent orientation. 
    Given a multicurve $\gamma \subset S$ and a ribbon graph $\Gamma$ with topological type $S \setminus \gamma$, the pair $(\gamma,\Gamma)$ is \textbf{jointly orientable} if each component of $\Gamma$ can be oriented in a way that induces a  consistent orientation on $\gamma$.
\end{definition}
Clearly, if $(\gamma, \Gamma)$ is jointly orientable, then each component of $\Gamma$ is orientable.
On the other hand, $(\gamma, \Gamma)$ need not be jointly orientable even if each component of $\Gamma$ can be oriented. See Figure \ref{fig: not_joint_orient}.

\begin{figure}
    \centering
    \includegraphics[width=0.75\linewidth]{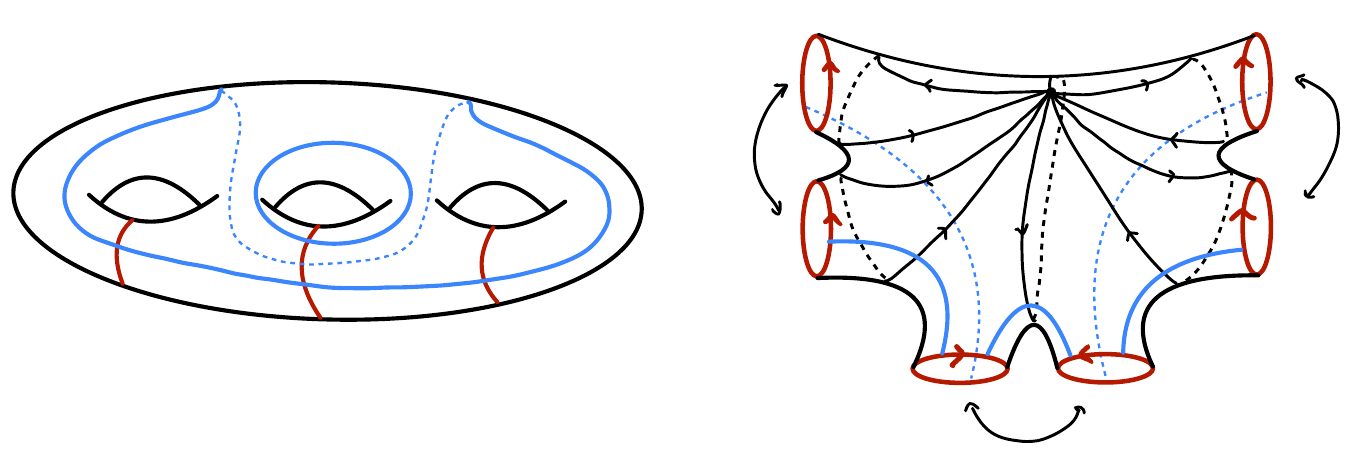}
    \caption{Left: a pair of multicurves on a surface.
    Right: An orientable ribbon graph $\Gamma$ with topological type $S\setminus \gamma$ where $(\gamma, \Gamma)$ is not jointly orientable.
    If one builds a square-tiled surface whose vertical and horizontal cylinders are the pair of multicurves on the left, then its graph of horizontal separatrices is the one shown on the right.}
    \label{fig: not_joint_orient}
\end{figure}

\begin{lemma}\label{lem: co-orient}
Let $q$ be a completely horizontally periodic quadratic differential with horizontal multicurve $\gamma$ and let $\Gamma$ denote the ribbon graph consisting of its horizontal separatrices (neither the height of the curves of $\gamma$ nor the length of the edges of $\Gamma$ are relevant). 
Then $q$ is globally the square of an abelian differential if and only if $(\gamma,\Gamma)$ is jointly orientable.
\end{lemma}
\begin{proof}
Consider the orientation double cover $\hat q \to q$.
We will show that, for any component $\Sigma \subset S\setminus \gamma$, the preimage of $\widehat\Sigma$ in $\hat q$ is disconnected if and only if the component of $\Gamma$ corresponding to $\Sigma$ is orientable, and $\widehat S$ is disconnected if and only if $(\gamma, \Gamma)$ is jointly orientable.

By construction, each curve of $\gamma$ can be realized as a core curve of a horizontal cylinder of $q$. Cutting along these core curves results in a flat cone surface with boundary homeomorphic to $S \setminus \gamma$.
This (possibly disconnected) surface now deformation retracts onto $\Gamma$ by collapsing the segments of its vertical foliation.

From this description, it follows that each component of $\Gamma$ is orientable if and only if the restriction of the horizontal foliation $\Im(q)$ to the corresponding component of $S \setminus \gamma$ is orientable (see also the discussion in Section 7.2 of \cite{shshI}), which happens if and only if the preimage $\widehat \Sigma$ is disconnected. Similarly, a joint orientation on $(\gamma, \Gamma)$ is equivalent to a global choice of orientation of $\Im(q)$ on $S$, which implies the holonomy double cover is disconnected.
\end{proof}

For later use, we record how twist tori and their associated measures interact with the $P$ action.

\begin{lemma}\label{lem:SL2Rtorimeasures}
For any $\gamma$, $\rg$, and $\bfh$, the twist torus measure $\nu_{\gamma}(\rg, \bfh)$ is $U$-invariant and satisfies
\begin{equation}\label{eqn:pushtorigeo}
(g_t)_* \nu_{\gamma}(\rg, \bfh) =
\nu_{\gamma}(e^t \rg, e^{-t} \bfh).
\end{equation}
\end{lemma}
\begin{proof}
The action of the horocycle flow can be written as the sum of cylinder twists on $\gamma_i$ of total length $h_i$; since Lebesgue measure on a torus is translation invariant, we see that twist torus measures are $U$-invariant.

As the Teichm{\"u}ller geodesic flow stretches all horizontal data and shrinking all vertical data, it takes twist tori to twist tori. To see that the measures are taken to one another, we observe that the twist torus measure is defined via Lebesgue measure on the twists, which constitute horizontal data.
Therefore, the Teichm{\"u}ller geodesic flow uniformly expands the (non-normalized) Lebesgue measure on the twist torus, hence preserves the normalized Lebesgue measure.
\end{proof}

Similarly, for a given horosphere
$\cP H_{\gamma} (\bfh)$, 
we define the \textbf{extremal length horosphere}
\[\cQ H_{ \gamma} (\bfh) = \cO(\cP H_{ \gamma} (\bfh)) \subset \QoT_g,\]
which is actually a lift of the $1$-level set of the extremal length function $L_{\bfh \gamma}: \T_g \to \RR_{>0}$ \cite{shshI,GM}.
Using Theorem \ref{thm:ortho homeo}, $\cQ H_{ \gamma}(\bfh)$ is a leaf of the  
unstable foliation.
For any facet of $\TRG(S \setminus \gamma)$ we can also define the \textbf{extremal length horosphere slice}
\[\cQ H_{ \gamma} (\bfh,\Gamma) = \cO(\cP H_{ \gamma} (\bfh,\Gamma)).\]
By construction of $\cP H_{ \gamma}(\bfh,\Gamma)$ and Theorems \ref{thm:ortho homeo} and \ref{thm:conjugacy}, $\cQ H_{ \gamma}(\bfh,\Gamma)$ is actually a leaf of the unstable foliation in the ambient stratum containing it.
Observe that, just as in the hyperbolic setting, $\cQ H_{ \gamma}(\bfh,\Gamma)$ is foliated by twist tori.

\subsection{Affine invariant subvarieties}\label{sec: AISs}
Lemma \ref{lem:SL2Rtorimeasures} tells us that we may  view expanding twist tori as pushforwards of a single torus along the geodesic flow.
By averaging these measures, we can invoke the seminal work of Eskin, Mizkakhani, and Mohammadi on the $P$- and $\PSL_2 \RR$-invariant measures and orbit closures in strata of holomorphic differentials. 
We briefly summarize their results below.
\medskip

An \textbf{affine invariant subvariety} (or \textbf{AIS}) $\AIS \subset \QM_g$ is a $\GL^+_2\RR$-invariant immersed sub-orbifold that is locally cut out by $\CC$-linear equations defined over $\RR$ in period coordinates.
Its unit-area locus $\AIS^1$ is $\PSL_2\RR$-invariant, and comes equipped with an ergodic $\PSL_2\RR$-invariant probability measure $\nu_{\AIS}^1$.
This measure has the property that, if $m$ is Lebesgue measure on $\RR$,  then $dm \, d\nu_\AIS^1$ is locally an affine linear measure on the defining linear subspaces in period coordinates; see \cite[Definition 1.1]{EM}.
While we will not use it directly, we mention that Filip proved that every AIS is an algebraic variety \cite{Filipvarieties}.

The following combines \cite[Proposition 1.3 and Theorem 1.4]{EM} and \cite[Theorems 2.1 and 2.10]{EMM}.

\begin{theorem}\label{thm:EMM}
Let $\cQ^1$ be any component of any stratum of $\QoM_g$.
\begin{enumerate}
    \item (Measure classification): Every $P$-invariant ergodic measure on $\cQ^1$ is also $\PSL_2\RR$-invariant and is $\nu_{\AIS}^1$ for some AIS, of which there are at most countably many.
    \item (Orbit closures): For any $q \in \cQ^1$, the orbit closure $\overline{P \cdot q} = \overline{ \PSL_2\RR \cdot q}$ is (the unit area locus of) an AIS.
    \item (Genericity): For any $q \in \cQ^1$, the orbit $P\cdot q$ equidistributes in its closure $\AIS^1$. That is, for any $f \in C_c(\cQ^1)$ and any $r \in \mathbb{R} \setminus \{0\}$,
\begin{equation*}\label{eqn:QMg P generic}
\lim_{T \to \infty}
\frac{1}{T} \int_0^T \int_0^r 
f\left( g_t u_s (q) \right)
\, ds \, dt
=
\int_{\AIS^1} f \,d\nu_{\AIS}^1.
\end{equation*} 
\end{enumerate}
\end{theorem}

Applying Theorem \ref{thm:EMM} to the setting of Lemma \ref{lem:SL2Rtorimeasures}, we deduce that the time-averaged pushforwards of a twist torus equidistribute to its $P$--orbit closure. 

\begin{proposition}\label{prop:avgtoriequid}
For any multicurve $\gamma$ on $S$,
any ribbon graph spine $\rg \in \TRG(S \setminus \gamma)$,
and any cylinder heights $\bfh \in \RR_{>0}^{n}$ with $\bfl \cdot \bfh = 1$, we have that
\[\frac{1}{T}\int_0^T (g_t)_* \left(\nu_{\gamma}(\rg, \bfh)\right) \, dt \to \nu_{\AIS}^1\]
as $T \to \infty$, where $\nu_{\AIS}^1$ is the unique affine probability measure supported on
$\AIS^1 := \overline{P \cdot \cQ\TT_{\gamma}(\rg, \bfh)}.$
\end{proposition}

\begin{proof}
Since $\gamma$, $\rg$, and $\bfh$ are fixed, throughout the proof we will denote $g_t \cdot  \cQ\TT_{\gamma}(\rg, \bfh)$ by $\cQ\TT^t$ and the associated twist torus measure on $\cQ\TT^t$ by $\nu^t$.

We first consider the (generic) situation in which the moduli $h_i/\ell_i$ of the horizontal cylinders are linearly independent over $\QQ$.
In this case, the torus $\cQ\TT^0$ is a minimal set for the horocycle flow and so the $U$ orbit closure of any point is the entire torus.
We therefore see that
\[\overline{ P \cdot q} = \overline{ P \cdot \cQ\TT^0} = \AIS^1\]
for every $q \in \cQ\TT^0.$

Invoking Theorem \ref{thm:EMM}, this implies that the F{\"o}lner sets $\{ g_t u_s q : t \in [0, T], s \in [0, r]\}$ based at any $q \in \cQ\TT^0$ equidistribute to $\nu_{\AIS}^1$ as $T \to \infty$. Integrating against the twist torus measure and applying Lemma \ref{lem:SL2Rtorimeasures} yields
\[\frac{1}{T} \int_0^T \int_0^r \int_{\cQ\TT^0}
f( g_t u_s q) \, d \nu^0(q) \, ds \, dt
=
\frac{1}{T} \int_0^T \int_{\cQ\TT^t}
f(q) \, d \nu^t(q) \, dt
\to
\int_{\AIS^1} f \, d\nu_{\AIS}^1\]
for every $f \in C_c(\cQ^1)$, where $\cQ^1$ is unit area locus of the ambient stratum. This is what we wanted to show.\bigskip

In general, the twist torus $\cQ\TT^0$ may not be minimal for the horocycle flow. In this case, we show that the $P$--orbit closure of $\nu^{0}$--almost every point is $\AIS^1$; the rest of the argument then proceeds as above. 

On any $q \in \cQ\TT^0$, choose a basis for relative homology consisting of horizontal saddle connections and saddles which connect the boundary components of the cylinders of $q$; this gives a local period coordinate chart for the entire stratum $\cQ$ (not just its unit area locus). 
Fix a finite collection\footnote{
In fact, one can cover a lift of the twist torus to $\QT_g$ with a single chart of this form, but we use a finite collection of restricted charts so that we do not need to worry about the $\Stab(\gamma)$ action.}
of such charts that covers $\cQ\TT^0$.
Let $\cN \subsetneq \AIS$ be any affine invariant subvariety (again, not just its unit area locus) not containing $\cQ\TT^0$.
In any of the period charts constructed above, $\cQ\TT^0$ is cut out by affine equations of the form $\hol(s)=\ell_i \in \RR$ (specifying the length of the horizontal saddle connections) and equations of the form $\Im(\hol(s)) = h_i$ (specifying the heights of the cylinders); compare \S\ref{sec: flat tt and horo}.
On the other hand, in period coordinates $\cN$ is a union of finitely many linear subspaces, each of which is cut out by linear equations with real coefficients. Therefore, we see that in each chart, 
$\cQ\TT^0 \cap \cN$ is a union of finitely many proper subspaces, hence has $\nu^0$ measure $0$. Moreover, we note that each subspace must be tangent to a proper subtorus of $\cQ\TT^0$, else its closure would be the entirety of $\cQ\TT^0$.

As the $P$--orbit closure of any $q$ is always an AIS and there are only countably many such (Theorem \ref{thm:EMM}.1), we see that 
\[\{q \in \cQ\TT^0 : \overline{Pq} \neq \AIS^1\}\]
is a union of countably many proper subtori of $\cQ\TT^0$. In particular, this set has measure zero with respect to $\nu^0$, and so the $P$--orbit closure of almost every point of $\cQ\TT^0$ is $\AIS^1$.
\end{proof}

\begin{remark}\label{rmk: subtori}
We note that the above proof also works for any subtorus $\cQ\TT' \subset \cQ\TT^0$; the point is just that the intersection of any AIS $\cN$ that does not contain $\cQ\TT'$ with the torus is always a union of subtori.
Thus the generic point of $\cQ\TT'$ equidistributes under the $P$ action to $\overline{ P \cdot \cQ\TT'}$.
\end{remark}

\subsection{Individual pushforwards}\label{subsec: individual pushforwards}
Forni proved that whenever one has a nice $\SL_2\RR$ action and equidistribution of the form above, one can remove the $T$-average in the above theorem at the cost of removing a small set of times. We recall that a set $Z \subset \RR$ is said to have {\bf zero (upper) density} if 
\[\limsup_{n \to \infty} \frac{\text{Leb}( Z \cap [0,n]) }{n} = 0\] 
and that a sequence $(a_t)$ converges as $t \to \infty$ outside of $Z$ if it converges along all sequences tending to $\infty$ that are eventually disjoint from $Z$.

\begin{theorem}[Theorem 1.1 of \cite{Forni}]\label{thm:Forni}
Let $\cQ^1$ be a stratum of unit area quadratic differentials and suppose that $\nu$ is a $U$-invariant probability measure on $\cQ^1$. If the limit
\[\nu_\infty = \lim_{T \to \infty} \frac1T \int_0^T (g_t)_* \nu \, dt \]
exists and is $U$-ergodic, then there exists a set $Z \subset \RR$ of zero density such that $(g_t)_* \nu \to \nu_{\infty}$ as $t \to \infty$ outside of $Z$.
\end{theorem}

\begin{remark}
    Forni conjectures that $Z$ is empty in the case that $\nu_\infty$ is $P$-invariant.
\end{remark}

Applying this to our context, we immediately get the following strengthening of Proposition \ref{prop:avgtoriequid}:

\begin{corollary}\label{cor:densetoriequid}
With all notation as in Proposition \ref{prop:avgtoriequid},
there is a set $Z \subset \RR$ of zero density such that 
\[(g_t)_*\nu_{\gamma}(\rg, \bfh) \to \nu_{\AIS}^1\]
as $t \to \infty$ outside of $Z$.
\end{corollary}

Forni also proved that if $\nu$ is especially nice, then the set $Z$ is empty. Recall that a {horospherical measure} $\nu$ is one that is supported in a leaf of the unstable foliation of $g_t$ on a stratum $\cQ^1\subset \QoM_g$, is absolutely continuous with respect to the canonical affine measure on the leaf, has continuous density, and has conditional measures along horocycle orbits equal to the restriction of one-dimensional Lebesgue measures.

\begin{proof}[Proof of Theorem \ref{thm: horospheres equi}]
From the definitions and the fact that $\cO$ is a piecewise real-analytic homeomorphism on unstable leaves that takes earthquake flow to horocycle flow in a time-preserving fashion (Theorems \ref{thm:ortho homeo} and \ref{thm:conjugacy}),
we see that if  $\sigma$ is a horospherical probability measure supported on $\cP H_{ \gamma}(\bfh)$ (in the sense of \S\ref{sec:horospheres}), then $\cO_*\sigma$ is horospherical with support $\cQ H_{ \gamma}(\bfh)$ (in the sense above).
Similarly, if $\sigma$ is a horospherical slice probability measure supported on $\cP H_{ \gamma}(\bfh,\Gamma) $, then $\cO_*\sigma$ is supported on $\cQ H_{ \gamma}(\bfh,\Gamma)$ and is horospherical for the ambient stratum component. 

Theorem 1.6 of \cite{Forni}\footnote{
Theorem 1.6 in Forni's paper is stated for the stable foliation, but clearly applies to the unstable.} (or \cite[Proposition 3.2]{EMM:effscc}) ensures $(g_t\circ\cO)_* \sigma$ converges weak-$*$ as $t\to \infty$ to the Masur--Smillie--Veech probability measure $\nu_\infty^1$ on the ambient stratum component of $\QoM_g$ containing the horosphere (slice).
Applying Theorem \ref{thm: pull back equidistribution} then gives that $(\Dil_t)_*\sigma$ converges weak-$*$ to $\cO^*\nu_{\cQ}^1$.
\end{proof}

\section{Twist tori built from pants decompositions}\label{sec:identifypants}

In this section, we identify the orbit closures of twist tori built out of pants decompositions, proving Theorem \ref{mainthm}.
The key tools are some properties of affine invariant subvarieties 
and Wright's notion of (cylinder) rank from \cite{Wright:cylinder}, discussed below. Throughout this section and the next, except when we discuss equidistribution, we will always work with an entire stratum/AIS, not just its unit area locus.

\subsection{Cylinders and rank}\label{subsec:cylrank}
Let $\mathcal N$ be an AIS in a stratum of (squares of) abelian differentials $\mathcal H = \cQ\M_g(\sing; +1)$.
Then $\cN$ is an immersed submanifold of $\mathcal H$ that is locally a $\CC$-linear subspace in period coordinates for $\mathcal H$ defined over $\RR$ (see \S\ref{sec: AISs}).
For $\omega \in \cN$, we have 
\[T_{\omega}\cN \le T_{\omega} \mathcal H\cong H^1(\omega, Z(\omega);\CC).\]
There is a natural linear projection \[p: H^1(\omega, Z(\omega);\CC) \to H^1(\omega;\CC).\]

\begin{definition}\label{def: rank}
    The \textbf{rank} of $\cN$ is the half the (complex) dimension of $p(T_{\omega}\cN)\le H^1(\omega;\CC)$.
The \textbf{rel} of $\cN$ is the (complex) dimension of $T_{\omega}\cN  \cap \ker p$.
\end{definition}


Rank and rel are among the most important invariants of an AIS.
The proof of the following uses only the definitions and basic properties.

\newcommand{\rank}{\textrm{rank}}
\newcommand{\rel}{\textrm{rel}}
\begin{lemma}\label{lem: rank plus rel}
    Suppose $\cN$ and $\cN'$ are AISs in a stratum of squares of abelian differentials.  Assume that  $\cN\subset \cN'$ and 
    \[\rank(\cN) + \rel (\cN) \ge \rank (\cN')+\rel(\cN'). \]
    Then $\cN = \cN'$.
\end{lemma}

\begin{proof}
Let $\omega \in \cN$ be a smooth point.\footnote{It is possible that $\cN$ is entirely contained in the orbifold locus of $\cH$. In this case, by ``smooth'' we mean a point whose local orbifold group is the smallest possible. See, e.g., \cite[\S2.3]{CSW}. In any case, our arguments only use the dimension of the relevant spaces, not their vector space structure.}
Then $T_\omega\cN\le T_\omega\cN'.$
Directly from this containment, we deduce
\[\rank (\cN)\le \rank (\cN') \text{ and }\rel(\cN)\le \rel(\cN').\]
    By hypothesis and the fact that rank and rel are never negative, we conclude
    \[\rank(\cN) = \rank(\cN')\text{ and } \rel(\cN) = \rel(\cN').\]
    The (complex) dimension of $\cN$ is computed as 
    \[\dim(\cN) = 2\rank(\cN) +\rel(\cN),\]
    and similarly for $\cN'$.
    Thus $\dim(\cN) = \dim(\cN')$, which implies that $\cN = \cN'$ near $\omega$.  By $\GL(2,\RR)$-invariance, $\cN = \cN'$ globally.
\end{proof}

Wright introduced the notion of rank in \cite{Wright:cylinder} and related it to the dimension of the space spanned by the directions tangent to certain \textbf{cylinder deformations}, which we recall below. 

For a given $\omega\in \cN\subset  \cH$, let $\mathcal C(\omega)$ be the set of maximal, singularity free horizontal cylinders in $\omega$.  Each element $C_i \in \mathcal C(\omega)$ is an isometrically embedded open Euclidean cylinder with positive height, foliated by closed, nonsingular oriented leaves of the imaginary foliation of $\omega$ which are all homotopic to its core curve $\gamma_i$.
Recall that  
\[u_s = \begin{pmatrix}
1 & s\\
0 & 1
\end{pmatrix} \in \SL_2\RR,\]
and define $u_s^{C_i}q\in \cH(\sing)$ to be the abelian differential obtained by applying the horocycle flow $u_s$ to $C_i$ but not the rest of $\omega$.
Denote by $\eta_i \in T_\omega \mathcal H(\sing)$ the tangent vector to the path $s\mapsto u_{s}^{C_i}\omega$ at $s=0$, and observe that $\eta_i \in H^1(\omega, Z(\omega);\RR)$. In fact, $\eta_i$ is Poincar\'e--Lefschetz dual to the homology class in $H_1(\omega\setminus Z(\omega); \RR)$ defined by $h_i\gamma_i$, where $h_i$ is the height of the cylinder $C_i$; see \cite[\S4]{MW_fullrank}.

\newcommand{\Twist}{\mathrm{Twist}}
The fact that $\cN$ is cut out by $\RR$-linear equations implies $T_\omega\cN = T_\omega^\RR\cN\otimes \CC$, where $T_\omega^\RR\cN\le H^1(\omega,Z(\omega); \RR)$.
The \textbf{twist space} of $\cN$ at $\omega$ is spanned by those cylinder deformations which remain in $\cN$:
\[\Twist(\omega, \cN):= \langle \eta_i : C_i \in \mathcal C(\omega)\rangle \cap T_\omega^\RR\cN. \]
We deduce the following from 
\cite[Section 8]{Wright:cylinder}. 

\begin{proposition}\label{prop: dimension of twist}
    Let $\cN$ be an AIS in a stratum of squares of abelian differentials  and suppose $\omega \in \cN$ is horizontally periodic. Then 
    \[\dim \Twist(\omega,\cN) \le \rank(\cN) + \rel(\cN),\]
    with equality if no other $\omega' \in \cN$ has more cylinders than $\omega$. 
\end{proposition}

\begin{proof}
    Corollary 8.11 of \cite{Wright:cylinder} asserts that, for any $\omega \in \cN$,
    \[\dim T_\omega^{\RR}\cN - \dim \Twist(\omega,\cN) \ge \rank(\cN).\]
    Combining Lemma 8.6 and Lemma 8.8 of \cite{Wright:cylinder} gives equality if $\omega$ has the maximal number of horizontal cylinders among all $\omega' \in \cN$.
    Using $\dim T_\omega^\RR\cN=2\rank(\cN)+\rel(\cN)$, we obtain the desired conclusion.
\end{proof}

Lemma 8.11 of \cite{Wright:cylinder}, used in the proof of Proposition \ref{prop: dimension of twist}, relies on the fact that $p(T_\omega^\RR\cN)$ is a symplectic subspace of $H^1(\omega;\RR)$ with respect to the intersection pairing \cite[Theorem 1.4]{AEM:symplectic}.  The rank can then be interpreted as the (real) dimension of a Lagrangian subspace of $p(T_\omega^\RR\cN)$.
Since $\eta_i$ is Poincar\'e--Lefschetz dual to $h_i\gamma_i$, and the curves $\{\gamma_i\}$ are all disjoint, one deduces that $p(\Twist(\cN, \omega))$ is isotropic.   We record the following consequence for use later; see also \cite[Section 8]{Wright:cylinder}.

\begin{proposition}\label{prop:twistspace lowerbd}
Let $\omega \in \cN$ be horizontally periodic. Then
\[\dim_{\RR} p(\Twist(\omega, \cN)) \le \rank(\cN).\]
\end{proposition}

Combining Lemma \ref{lem: rank plus rel} and Proposition \ref{prop: dimension of twist} allows us to easily identify orbit closures of twist tori associated to pants decompositions. 

\begin{corollary}\label{cor: identify AIS}
	Let $\cQ$ be a component of a stratum of holomorphic quadratic differentials on a closed genus $g$ surface and let $\AIS$ be an AIS. Suppose that $\AIS$ contains a surface $q$ with $3g-3$ horizontal cylinders and that all cylinder deformations of these horizontal cylinders remain in $\cL$.  Then $\cL = \cQ$.     
\end{corollary}

\begin{proof}
No holomorphic quadratic differential on a genus $g$ surface can have more than $3g-3$ horizontal cylinders, as their core curves must all be disjoint and homotopically distinct.
Thus $q$ contains as many horizontal cylinders as possible for any surface in $\cQ$ (hence for any surface in $\AIS$).
By hypothesis, all of these cylinders can be twisted individually to stay in $\AIS$.

The loci $\widehat \AIS \subset \widehat \cQ$ of holonomy double covers of surfaces in $\AIS$ and $\cQ$, respectively, are AISs in a stratum of abelian differentials. 
Every flat cylinder on $q$ lifts to two cylinders on $\hat q$,\footnote{This can be seen because the holonomy of a curve is the same as its winding number mod 2 with respect to the horizontal line field.}
hence the holonomy double cover $\hat q$ of $q$ has the most horizontal cylinders among all surfaces in $\widehat \cQ$.
The cylinder deformations tangent to $\cL$ at $q$ lift to (sums of) cylinder deformations tangent to $\widehat \cL$ at $\hat q$, 
so
\[\dim \Twist(\hat q, \widehat \AIS) = 3g-3\]
(dimension is preserved because the kernels of transfer maps on (co)homology are always torsion).
Using Proposition \ref{prop: dimension of twist} twice, once for $\widehat \AIS$ and once for $\widehat \cQ$,
\[\rank (\widehat \AIS) + \rel (\widehat \AIS) = 3g-3 = \rank (\widehat \cQ) +\rel (\widehat \cQ).\]
By Lemma \ref{lem: rank plus rel}, we conclude that $\widehat \AIS = \widehat \cQ$, hence that $\AIS = \cQ$, establishing the result.
\end{proof}

\subsection{Proof of the main theorem}\label{subsec:proof of mainthm}
At this point, we can assemble and prove a more detailed version of our main Theorem \ref{mainthm}.  For the reader's convenience, we quickly recall the setup and the statement below.

Given a pants decomposition $\gamma$ of a topological surface $S$ and a length vector $\bfl \in\RR_{>0}^{3g-3}$.
Suppose $X$ is a hyperbolic surface glued together out of geodesic pants curves with lengths given by $\bfl$, and let $\rg = \Sp(X\setminus \gamma)$ be the metric ribbon graph from \S\ref{subsec: arcs and rgs}.
Let $\nabla_\gamma(\bfl)$ denote the number of components of $S \setminus \gamma$ that have property \hyperref[a+b=c]{$(\nabla)$} from the introduction, i.e., such that the lengths of two of its boundaries sum to the length of the third.
This corresponds to the corresponding component of the ribbon graph $\rg$ having a 4-valent vertex and being orientable (in the sense of Definition \ref{def: orientable}).

Recall our notation that if $\cQ$ is a component of a stratum of quadratic differentials, then  $\nu_\cQ^1$ is the Masur--Smillie--Veech probability measure on its unit area locus $\cQ^1$ (\S\ref{subsec:horocycle flow}).

\begin{theorem}\label{thm: mainthm}
Fix any pants decomposition $\gamma$, any $\bfl \in \RR^{3g-3}_{>0}$, and any $\bfh \in \RR^{3g-3}_{>0}$ such that $\bfl \cdot \bfh = 1$.
There is a set $Z \subset \RR$ of zero density such that 
\[ \lim_{\substack{t \to \infty \\ t \notin Z}}
\mu_\gamma(e^t \bfl, e^{-t} \bfh) = \cO^*\nu_\cQ^1,\]
where $\cQ$ is the (connected) stratum $\QM_g(1^{4g-4-2\nabla_\gamma(\bfl)}, 2^{\nabla_\gamma(\bfl)}; \varepsilon)$, and $\varepsilon = 1$ if and only if $\nabla_\gamma(\bfl) = 2g-2$ and $(\gamma, \rg)$ is jointly orientable.
\end{theorem}

\begin{proof}[Proof of Theorem \ref{mainthm} given Theorem \ref{thm: mainthm}]
    We only need to justify the two numbered items.
    Mirakhani proved that $\cO^*\nu_\cQ^1 = \mu_{\Mirz}/b_g$ when $\cQ$ is the principal stratum of quadratic differentials \cite[Theorems 1.3 and 1.4]{MirzEQ}.
    This gives item (1) of Theorem \ref{mainthm}.

    For item (2), note that for any stratum component $\cQ$, the measure $\nu_\cQ^1$ is ergodic for the $P$-action. 
    Any two distinct components of strata yield distinct Masur--Smillie--Veech measures. By ergodicity and Theorem \ref{thm:conjugacy}, if $\cQ$ is not the principal stratum, then $\cO^*\nu_\cQ^1$ and $\mu_{\Mirz}$ are mutually singular. 
\end{proof}

\begin{proof}[Proof of Theorem \ref{thm: mainthm}]
Form the twist tori $\cP\TT_\gamma(\mathsf{Y}, \bfh)$ and $\cQ\TT_\gamma(\mathsf{Y}, \bfh)$ as in \S\S\ref{subsec: hyp tt} and \ref{sec: flat tt and horo}, as well as their associated measures $\mu_\gamma(\mathsf{Y}, \bfh) = \mu_\gamma(\bfl, \bfh)$ and $\nu_\gamma(\mathsf{Y}, \bfh)$.
Note that the edge lengths of $\mathsf Y$ are linear functions of the lengths $\bfl$ (see \cite[Figure 1]{shshI}),  so rescaling $\bfl$ corresponds to rescaling $\rg$.
From \S\ref{sec: flat tt and horo} and Lemma \ref{lem:SL2Rtorimeasures}, recall that $\cO$ interchanges these tori and that $g_t$ takes twist tori to twist tori \eqref{eqn:pushtorigeo}:
\[\cO_*\mu_\gamma(\mathsf{Y}, \bfh)
=\nu_\gamma(\mathsf{Y}, \bfh)
\hspace{ 2ex}
\text{and}
\hspace{ 2ex}
(g_t)_* \nu_{\gamma}(\rg, \bfh) =
\nu_{\gamma}(e^t \rg, e^{-t} \bfh).\]

\begin{claim}\label{clm: identity twist tori pants}
With all notation as above, we have
\[\overline {P\cdot \cQ \mathbb T_\gamma(\rg,\bfh)} = \QoM_g(1^{4g-4-2\nabla_\gamma(\bfl)}, 2^{\nabla_\gamma(\bfl)};\varepsilon) \]
where $ \varepsilon = +1$ if $(\gamma,\rg)$ is jointly orientable, and $\varepsilon=-1$ if not.  
\end{claim}

\begin{proof}
The spine $\mathsf{Y}$ has $2g-2$ components, $\nabla_\gamma(\bfl)$ of which have a unique 4-valent vertex and the rest of which have two trivalent vertices.
Thus, we can identify the ambient stratum component containing the twist torus using two facts: that $k$-pronged singularities of $\cO_\gamma(X)$ correspond to zeros of order $k-2$ of $\cO(X,\bfh\cdot\gamma)$ (Theorem \ref{thm:conjugacy}) and that joint orientability of $(\gamma,\rg)$ is equivalent to $\cO(X,\bfh\cdot\gamma)$ being the square of an abelian differential (Lemma \ref{lem: co-orient}).
Connectivity of these strata is a consequence of the classification of connected components of strata \cite{KZstrata,Lanneau}.

The argument presented in the proof of Proposition \ref{prop:avgtoriequid} gives that $\overline{P\cdot \cQ \mathbb T_\gamma(\rg,\bfh)} = \overline{Pq_0}$ for generic $q_0 \in \cQ \mathbb T_\gamma(\rg,\bfh)$.  
By Theorem \ref{thm:EMM}, $\overline{Pq_0}$ is the unit length locus of an AIS $\AIS$, and since $\overline{Pq_0}\supset \cQ \mathbb T_\gamma(\rg,\bfh)$, Corollary \ref{cor: identify AIS} implies $\AIS$ must be the entire ambient stratum.
\end{proof}

By Corollary \ref{cor:densetoriequid} and \eqref{eqn:pushtorigeo}, we have that these flat twist tori equidistribute to the Masur--Smillie--Veech probability measure $\nu_{\cQ}^1$ on the ambient stratum outside a set $Z$ of zero density:
\[\nu_{\gamma}(e^t \rg, e^{-t} \bfh) \xlongrightarrow{t \to \infty, \, t \not \in Z}
\nu_{\cQ}^1.\]
Applying Theorem \ref{thm: pull back equidistribution}, we therefore get that
\[\mu_{\gamma}(e^t \rg, e^{-t} \bfh) =
\cO^*\nu_{\gamma}(e^t \rg, e^{-t} \bfh)
\xlongrightarrow{t \to \infty, \, t \not \in Z}
\cO^* \nu_{\cQ}^1.\]
This is what we wanted to show.
\end{proof}

\section{Other twist tori}\label{sec:identifyother}

One of the pivotal steps in our proof of Theorem \ref{mainthm} relied on the fact that no flat surface can have more than $3g-3$ horizontal cylinders (Corollary \ref{cor: identify AIS}), which makes it seem like pants decompositions are special.
In this section, we demonstrate that this phenomenon is actually much more general.
Using more tools from Teichm{\"u}ller dynamics, we show the equidistribution of more families of flat twist tori to their ambient strata (and to other AISs), which implies the equidistribution of other hyperbolic twist tori via Theorem \ref{thm: pull back equidistribution}.

Apart from the identification of orbit closures, the proofs of the results in this section are identical to that of Theorem \ref{mainthm} and are therefore omitted.
The examples in this section are meant to be illustrative, not exhaustive --- we invite the reader to import their favorite techniques to this setting.

\subsection{Regular plumbing fixtures}\label{subsec:plumbing}
As a motivating example, let us consider the following family of generalizations of our problem. We emphasize, however, that all of the computations in this section are much more general and apply to pushforwards by the dilation flow of any of the tori defined in Section \ref{sec: hyp twist tori}.

For $p \ge 3$ and $L>0$, we define the \textbf{regular plumbing fixture with $p$ holes} to be the unique hyperbolic metric on a $p$-holed sphere with totally geodesic boundaries of length $L$ and an (orientation-preserving) dihedral symmetry group of order $2p$.
Equivalently, this is the surface corresponding to the ribbon graph obtained by gluing together $p$ intervals of length $L/2$ at their vertices.
see Figure \ref{fig:regular plumbing}. 

\begin{figure}[h!]
    \centering
    \includegraphics[width=0.4\linewidth]{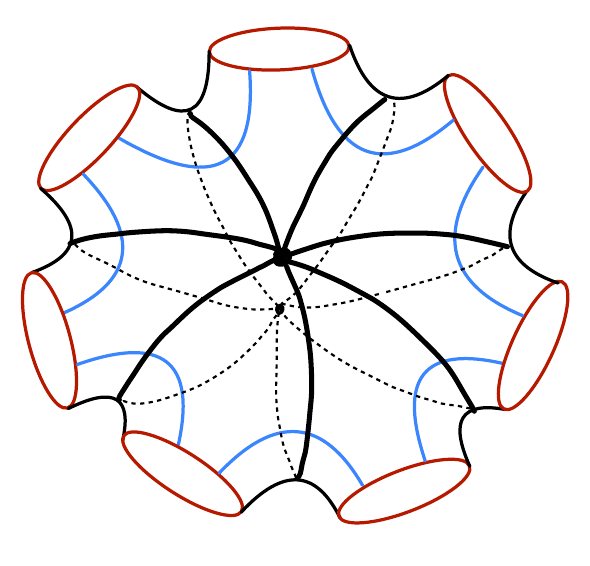}
    \caption{A regular plumbing fixture on a $7$-holed sphere.}
    \label{fig:regular plumbing}
\end{figure}

Suppose that $\gamma$ is a multicurve on $S$ such that each component $\Sigma_1, \ldots, \Sigma_k$ of $S \setminus \gamma$ has genus 0; let $p_i$ denote the number of boundary components of $\Sigma_i$.
Taking each component to be a regular plumbing fixture with boundary lengths $L$, we get a twist torus $\TT_\gamma(L) \subset \M_g$
whose corresponding metric ribbon graph spine $\mathsf{X} = \mathsf{X}(p_1, \ldots, p_k)$ has $k$ components, the $i^{\text{th}}$ of which has two vertices of valence $p_i$ and edges of length $L/2$.
Paralleling Conjecture \ref{conj:twist tori}, one can ask: what is the limiting distribution of these tori in $\M_g$ as $L \to \infty$?

For any choice of heights $\bfh$, the corresponding flat twist torus $\cQ \mathbb T_{\gamma}(\mathsf{X}, \bfh)$ parametrizes surfaces with cone points of angle
$p_1 \pi, p_1 \pi, p_2 \pi, p_2 \pi, \ldots, p_k \pi, p_k \pi$, so unless all $p_i=3$, the limiting distribution of the lifted hyperbolic twist tori $\cP \mathbb T_{\gamma}(\mathsf{X}, \bfh) \subset \PoM_g$ will not be Mirzakhani measure. Moreover, since the orbit closure of $\cQ \mathbb T_{\gamma}(\mathsf{X}, \bfh)$ is contained in a non-principal stratum, the possible limiting distributions on $\M_g$ may be qualitatively different from Mirzakhani measure. For example, the support of the measure may not be (coarsely) dense in $\PoM_g$ (compare \cite[Remark 1.3]{shshI}).

Whenever any $p_i$ is even, one can check that the differentials $q \in \cQ \mathbb T_{\gamma}(\mathsf{X}, \bfh)$ do not have the maximal number of cylinders for their ambient stratum.
For example, if all $p_i=4$ then $\# \gamma= 2g-2$, but we have seen in the previous section that there are differentials in the stratum $\cQ\M_g(2^{2g-2}; \varepsilon)$ with an entire pants decomposition worth of horizontal cylinders.
Thus we will need to employ different techniques.

\subsection{Jointly orientable constructions}\label{subsec:jo_construct}
We first focus on the case when $(\gamma, \Gamma)$ is jointly orientable, that is, when the twist torus lives inside a stratum of (squares of) abelian differentials.
Building on work of Mirzkakhani and Wright on ``full rank'' AISs \cite{MW_fullrank},
recent deep work of Apisa and Wright classifies the ``large'' AISs in these strata.

Following \cite{AW:high_rank}, say that an AIS $\AIS$ in a stratum of genus $g$ abelian differentials has {\bf high rank} if 
\[\rank(\AIS) \ge \frac{g}{2}+1.\]

\begin{theorem}[Theorem 1.1 of \cite{AW:high_rank}]\label{thm: high rank}
Any high rank AIS in a stratum of squares of abelian differentials is either a component of the stratum or the locus of all holonomy double covers of surfaces in a component of a stratum of quadratic differentials.
\end{theorem}

Given any twist torus $\cQ\TT_\gamma(\rg, \bfh)$ contained in such a stratum of squares of abelian differentials, the dimension of the projection of the tangent space of the twist torus to absolute cohomology (as in Proposition \ref{prop:twistspace lowerbd}) is just the dimension of the span of the curves of $\gamma$ in homology.
Thus, so long as $\gamma$ has at least $g/2 + 1$ homologically independent curves, the $P$ orbit closure of $\cQ\TT_\gamma(\rg, \bfh)$ is the unit area locus in either the ambient stratum or a locus of double covers.

Combining this Theorem together with the argument from Section \ref{subsec:proof of mainthm}, {\em mutatis mutandis}, immediately yields a result for jointly orientable gluings of regular plumbing fixtures.

\begin{theorem}\label{thm: plumbingAD}
Let $\gamma$ be a multicurve on a surface $S$ such that $S \setminus \gamma$ is a union of $p_1, \ldots, p_k$-holed spheres, where all $p_i$ are even. 
Fix $L > 0$ and consider the twist tori in $\M_g$ obtained by gluing together $p_i$-holed regular plumbing fixtures with all boundary lengths $L$ along $\gamma$.
Suppose furthermore that $(\gamma, \mathsf{X})$ is jointly orientable, where $\mathsf{X} = \mathsf{X}(p_1, \ldots, p_k)$ is the ribbon graph spine for the regular plumbing fixtures.
Fix any choice of heights $\bfh$ such that $\sum h_i = 1/L$.
There is a set of times $Z$ of zero density such that
\[\mu_{\gamma}(e^t\mathsf{X},e^{-t}\bfh)
\xlongrightarrow{t \to \infty, \, t \not \in Z}
\cO^*\nu_{\AIS}^1
\text{ as measures on }
\PoM_g.\]
\begin{enumerate}
    \item If $S$ admits an involution fixing each of the curves of $\gamma$ setwise but reversing their orientations (hence $k\le 2$), then $\AIS$ is the hyperelliptic locus in the component of the stratum 
    \[\QM_g\left((p_1-2)^{2}, \ldots, (p_k-2)^2; +1\right).\]
    \item Otherwise, $\AIS$ is the entire component of the ambient stratum containing $\cQ\TT_{\gamma}(\mathsf{X}, \bfh)$.
\end{enumerate}
\end{theorem}

To identify the $P$-orbit closure of the twist tori in this theorem, it suffices to use the full rank result of Mirzakhani--Wright \cite{MW_fullrank}, rather than Theorem \ref{thm: high rank}.

\begin{proof}
We note that since $\gamma$ cuts $S$ into surfaces without genus, it spans a $g$-dimensional subspace of homology. In particular, the orbit closure of the twist torus is high (in fact, full) rank.
The only thing that remains to be justified is why no other loci of double covers of quadratic differentials can appear.
For this to be the case, $\gamma$ must be setwise invariant under some orientation-preserving involution $\iota$ of $S$.
Unless $\iota$ is hyperelliptic, there are two curves of $\gamma$ which must be interchanged. Twisting those curves independently is not $\iota$-invariant, hence the tangent space to the twist torus is not $\iota$-invariant.
\end{proof}

This proof is clearly much more general: one can pick any metric ribbon graph $\mathsf{X}$ such that $(\gamma, \mathsf{X})$ is jointly orientable and get the same result.
See Figure \ref{fig:minimal stratum} for examples that land in the minimal strata.

\begin{remark}
Unlike Theorem \ref{mainthm}, the gluing pattern of $\gamma$ can determine the limiting measure, as some strata $\cQ\M_g(\sing; +1)$ are disconnected. 
However, one can read off the parity of the induced spin structure from the data of $\gamma$ and $\mathsf{X}$ by computing winding numbers.
\end{remark}

\begin{figure}[h]
    \centering
    \includegraphics[width=\linewidth]{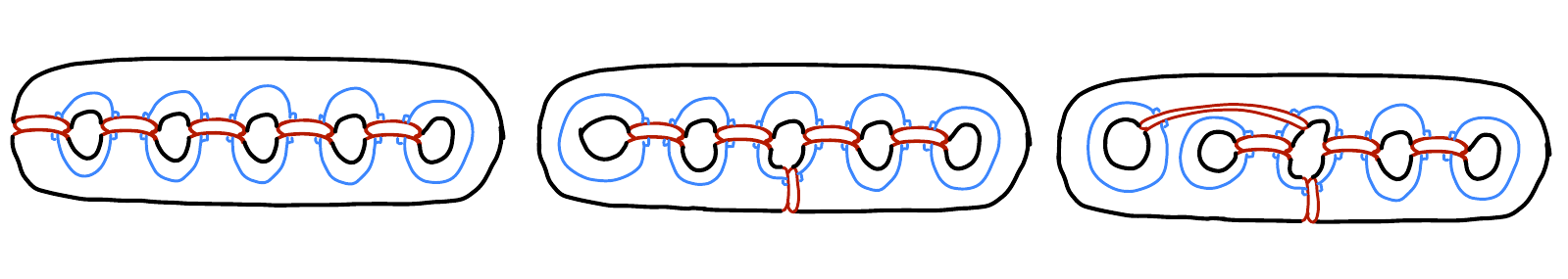}
    \caption{Twist tori corresponding to the minimal stratum of squares of abelian differentials. The drawn arc systems are dual to ribbon graphs each with a unique vertex of valence $4g-2$. 
    The leftmost torus lands in the hyperelliptic locus, while the other two land in the odd and even spin components (which diagram corresponds to which spin depends on the parity of $g$ mod $4$, see \cite[\S6.3]{strata2}).}
    \label{fig:minimal stratum}
\end{figure}

Using more specific results from Teichm{\"u}ller dynamics, we can identify the limiting distributions of more twist tori.
We mention in particular the following result of Winsor \cite{Winsor}, which itself admits further generalizations.

\begin{theorem}[Theorem 1.3 of \cite{Winsor}]\label{thm:Winsor}
Suppose $\AIS$ is an AIS in a component $\cH$ of a stratum of squares of abelian differentials.
Suppose further there is some $\omega \in \AIS$ with a set of three horizontal cylinders whose core curves bound a pair of pants, and that cylinder deformations on each of these cylinders remains in $\AIS$. Then $\AIS = \cH$.
\end{theorem}

Consequently, if $\gamma$ is any multicurve such that some component of $S \setminus \Gamma$ is a pair of pants, and $\rg \in \T(S \setminus \gamma)$ is any metric ribbon graph such that $(\gamma, \rg)$ is jointly orientable, then the twist tori
$\cP\TT_\gamma(e^t \rg, e^{-t}\bfh)$ equidistribute outside a set of zero density to the pullback of Masur--Smillie--Veech measure coming from their ambient stratum.

\subsection{Lifting homologies}
When $(\gamma, \rg)$ is not jointly orientable we have fewer tools at our disposal, as the classification of orbit closures in quadratic strata is not as far along as the abelian case (for example, there is no known analogue of Theorem \ref{thm:Winsor}).
However, we can still apply Theorem \ref{thm: high rank} in many situations.

In order to do this, let us record some numerical invariants of quadratic doubles.
Let $\cQ$ be a component of a stratum $\QM_g(\sing;-1)$ of genus $g$ holomorphic quadratic differentials, where $\sing$ is a partition of $4g-4$.
Let $\singodd$ denote those elements of $\sing$ that are odd and likewise define $\sing_{\textrm{even}}$.

\begin{lemma}\label{lem: RH computations}
Let $\cQ\subset \cQ\M_g(\sing;-1)$ be a component of a stratum and let 
$\widehat \cQ \subset \cQ\M_{\widehat g}(\widehat \sing;+1)$
be the AIS consisting of holonomy double covers of surfaces in $\cQ$.
Then
    \[\widehat g = 2g +\frac12\#\singodd-1,
    \hspace{4ex}
    \rank(\widehat\cQ) = g+ \frac12 \#\singodd-1,
    \hspace{4ex}
    \text{and}
    \hspace{4ex}
    \rel(\widehat\cQ) = \#\sing_{\textrm{even}}.
    \]
\end{lemma}

\begin{proof}
    This is \cite[Lemma 4.2]{AWmarkedpts}, but we give some details for completeness; see also \cite[Theorem 26]{BW:symplectic} for a related computation in the context of train tracks.
    The expression for $\widehat g$ can be computed using the Riemann--Hurwitz formula for the holonomy double cover $\hat q \to q$ branched over the zeros of $q \in \QM(\sing;-1)$.
    To compute rank, subtract the complex dimension of the kernel of the map to absolute cohomology from the complex dimension of $\QM(\sing;-1)$ and divide by two.
    Since $\dim \QM(\sing;-1) = 2g-2 + \#\sing$ and the kernel of the map to absolute cohomology is generated by certain relative cycles associated to the even order zeros of any $q\in \cQ$, the formulas follow.     
\end{proof}
    
\begin{remark}\label{rmk: 6 odd zeros}
We note that the locus of double covers $\widehat \cQ$ is only high rank in $\cQ\M_{\widehat g}(\widehat \sing;+1)$ once $\singodd \ge 6.$
\end{remark}

Proposition \ref{prop:twistspace lowerbd} gives a lower bound for the rank of an AIS $\AIS$ in terms of cylinder deformations.
When $\AIS$ contains the entire twist torus $\cQ\TT(\rg, \bfh)$, this dimension can be computed explicitly.

\begin{proposition}\label{prop: number of relations}
Let $\cQ$ be a component of $\QM_g(\sing; -1)$.
Let $q \in \cQ$ be horizontally periodic with horizontal multicurve $\gamma$ and let $\Gamma$ denote the ribbon graph of horizontal separatrices on $q$. Then
\[\dim_{\RR} p(\Twist(\hat q, \widehat\cQ)) = \#\gamma - N_{o}(\Gamma),\]
where $N_{o}(\Gamma)$ is the number of orientable components of $\Gamma$.
\end{proposition}
\begin{proof}
The basic idea is that cylinder deformations are dual to their core curves, and each subsurface of $S\setminus \gamma$ defines a linear relation between its boundary curves in the homology of $S$.
These relations lift to relations between the curves of $\widehat \gamma$ in the homology double cover $\widehat S$ exactly when that subsurface is disconnected in $\widehat S$.
The main technical difficulty is just to ensure that there are no other unaccounted-for relations.

Let us begin by setting some notation.
Enumerate the curves of $\gamma$ as $\gamma_1, \ldots, \gamma_n$.
The linear holonomy along the core curve of a horizontal cylinder is trivial, so we know the preimage of each $\gamma_i$ in the holonomy cover $\hat q$ has two components.
Pick a lift $\overline \gamma_i \subset \hat q$, so the full preimage is
$\overline \gamma_i \sqcup \iota (\overline \gamma_i)$, where $\iota$ denotes the covering involution. 
Choose an orientation on each $\gamma_i \subset S$ and orient $\overline \gamma_i$ so that it maps with degree $1$ to $\gamma_i$.
Push that orientation forward by $\iota$ to induce an orientation on $\iota(\overline \gamma_i)$ and define 
\[\widehat\gamma_i = \overline \gamma_i -\iota_*\overline \gamma_i \in H_1(\hat q\setminus Z(\hat q); \ZZ).\]
Up to scale, these are the Poincar{\'e}-Lefschetz duals of the lifts of cylinder deformations of $q$ to $\hat q$, and satisfy $\iota_* \widehat \gamma_i = - \widehat \gamma_i$.
See Figure \ref{fig: orientation cover}.

The inclusion induces a natural map $i_*:H_1(\hat q \setminus Z(\hat q);\ZZ) \to H_1(\hat q;\ZZ)$ and duality makes the following square commute (with any coefficients):
\[
\begin{tikzcd}
    H^1(\hat q,Z(\hat q))^- \arrow[r,leftrightarrow] \arrow[d,"p"]
    & H_1(\hat q \setminus Z(\hat q))^- \arrow[d,"i_*"] \\
  H^1(\hat q)^- \arrow[r,leftrightarrow]
&  H_1(\hat q)^- \end{tikzcd}\]
Thus, in order to prove the Proposition it is equivalent to show
\[
\dim_\RR
i_*\langle \widehat\gamma_1, ..., \widehat\gamma_n\rangle = n - N_{co}.\]
For brevity, and to avoid confusion with the action $\iota_*$ of the covering involution on the homology of $\widehat S$, we will henceforth drop the $i_*$ and write $\widehat \gamma_i \in H_1(\widehat S;\RR)^-$.  Also, let $V =i_*\langle \widehat\gamma_1, ..., \widehat\gamma_n\rangle$.

Enumerate the subsurfaces $S\setminus \gamma = \Sigma_1\cup  ... \cup\Sigma_k$, and let $\widehat \Sigma_j$ be the full preimage of $\Sigma_j$ under $\widehat S \to S$.
Let $\Gamma_j$ denote the component of $\Gamma$ corresponding to $\Sigma_j$, and observe that a subsurface $\widehat \Sigma_j$ is disconnected if and only if $\Gamma_j$ is orientable (as in the proof of Lemma \ref{lem: co-orient}).
Consider annular neighborhoods $A_i$ of $\gamma_i \subset S$ and let $\widehat A_i$ be their full preimages.

Now consider the Mayer--Vietoris long exact sequence (with coefficients in $\RR$, say)
    \begin{equation}\label{eqn: long exact 1}
        0 \to H_2(\widehat S) \to \oplus_{i,j} H_1(\widehat A_i \cap \widehat \Sigma_j) \to \left( \oplus_{i =1}^n H_1 (\widehat A_i) \right) \bigoplus \left(\oplus_{j = 1}^k H_1(\widehat \Sigma_j)\right) \to H_1(\widehat S) \to ...
    \end{equation}
    By naturality of $\iota_*$, this long exact sequence is the sum of two long exact sequences corresponding to the $\iota$-invariants and the $\iota$-anti-invariants, and we are interested in the latter.

Since $\iota$ is orientation-preserving and $\widehat S$ is closed, we have $H_2(\widehat S)^-= 0$.
By definition, $H_1(\widehat A_i)^- = \langle \widehat \gamma_i \rangle $.
Let $\gamma_i^+$ and $\gamma_i^-$ be curves parallel to $\gamma_i$ obtained by pushing off to the right and left, respectively; see Figure \ref{fig: orientation cover}.
For any fixed $i$, we have that $\bigcup_{j} \widehat A_i \cap \widehat \Sigma_j$ is a union of four cylinders, the cores of which correspond to $\overline\gamma_i^{\pm}$ and $\iota_*\overline\gamma_i^{\pm}$.
Thus $\oplus_{i,j} H_1(\widehat A_i \cap \widehat \Sigma_j)^-$ is freely generated by $\{\widehat \gamma_i^\pm: i = 1, ..., n\}$.

\begin{figure}[ht]
        \centering
        \includegraphics[width=0.7\linewidth]{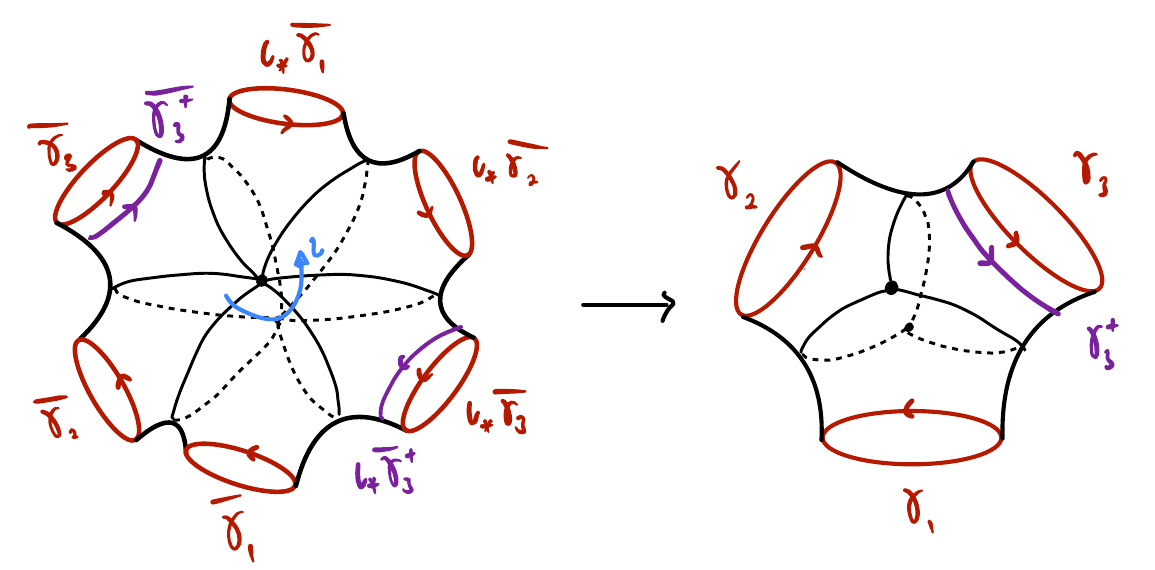}
        \caption{A non-trivial cover $\widehat \Sigma \to \Sigma$ corresponding to a non-orientable ribbon graph. We have $\widehat\gamma_i = \overline \gamma_i - \iota_* \overline \gamma_i$ and the push-off to the right is $\widehat \gamma_i^+ = \overline \gamma_i^+ - \iota_* \overline \gamma_i^+$.
        }
        \label{fig: orientation cover}
    \end{figure}
    
Since each $\widehat \Sigma_j$ is the total preimage of $\Sigma_j$, it is invariant under the covering involution $\iota_*$.
Hence the anti-invariance distributes over the sum in \eqref{eqn: long exact 1}, yielding
    \begin{equation}\label{eqn: long exact 2}
        0 \to  \langle \widehat\gamma_i^\pm \rangle  \xrightarrow{F} \left( \oplus_{i =1}^n \langle \widehat \gamma_{i} \rangle \right) \bigoplus \left(\oplus_{j = 1}^k H_1(\widehat \Sigma_j)^-\right) \xrightarrow{G} H_1(\widehat S)^- \to ...,
    \end{equation}
    where $F :\widehat\gamma_i^\pm \mapsto(\widehat \gamma_i, - \widehat \gamma_{i}^\pm)$ and $G$ sums the components together.
    Note $\dim (\im F )= 2n$ since \eqref{eqn: long exact 2} is exact.
    \medskip

    We now use this sequence to compute the dimension of $V$. We begin by observing that $\im F\subset G\inverse(V)$, so that $V\cong G\inverse (V) /\im F $. 
    From the definition of $V$, we see 
    $G\inverse(V)$ is generated by the $3n$ homology classes $\{\widehat\gamma_i^-, \widehat\gamma_i, \widehat\gamma_i^+\}_{i = 1}^n$.  The classes $\{\widehat\gamma_i^\pm\}_{i = 1}^n$ may satisfy some relations in $\oplus H_1(\widehat \Sigma_j)^-$, which we now compute.

    Each  $\widehat\Sigma_j$ is a surface with an even number $m_j$ of boundary components 
    \[\{c_1, ..., c_{m_j}\} \cup \{\iota_*c_1, ..., \iota_*c_{m_j}\},\]
    where $\{c_i\}\subset \{ \pm \overline \gamma_i^\pm\}$.
    We assume without loss of generality that the orientations are chosen so that 
    \[\sum c_i + \sum \iota_*c_i= 0 \in H_1(\widehat \Sigma_j)\]
    holds.
    If $\widehat\Sigma_j$ is connected, then this is the only non-trivial relation satisfied by $\{c_i\} \cup \{\iota_* c_i\}$, while if $\Sigma_j$ is disconnected, then  additionally
    \[\sum c_i = 0 = \sum \iota_* c_i \in H_1(\widehat \Sigma_j).\]
    In particular, 
    \[ \{c_i -\iota_* c_i: {i = 1, ..., m_j} \} \subset H_1(\widehat \Sigma_j)^- \]
    is a linearly independent set if and only if $\widehat \Sigma_j$ is connected and otherwise spans a subspace of dimension $m_j-1$.
    By definition, there are $N_{o}(\Gamma)$ disconnected $\widehat\Sigma_j$, so that $\dim G\inverse(V) = 3n- N_{o}(\Gamma)$.
    Putting this all together, $V \cong G\inverse (V) /\im F$ has dimension
    \[\dim G\inverse (V) - \dim (\im F) = (3n - N_{o}(\Gamma)) - 2n = n - N_{o}(\Gamma),\]
    completing the proof of the Proposition.
\end{proof}

Combining Lemma \ref{lem: RH computations} and Proposition \ref{prop: number of relations} gives a criterion for when an AIS $\AIS$ containing a twist torus must have high rank.
Pugging this into Theorem \ref{thm: high rank} yields the following:

\begin{corollary}\label{cor: high rank QD identification}
Let $\cQ \mathbb T_\gamma(\rg,\mathbf{h}) \subset \QM_g(\sing;-1)$ be a flat twist torus parametrizing quadratic differentials which are not squares of abelian differentials.
If
\[\# \gamma \ge g+\frac14 \#\singodd +N_{o}(\rg) + \frac{1}{2},\]
then the only AIS containing $\cQ \mathbb T_\gamma(\rg,\mathbf{h})$ is the ambient stratum component.
\end{corollary}

Applying this to our motivating example, we get that twist tori constructed out of regular plumbing fixtures generally equidistribute to the pullbacks of Masur--Smillie--Veech probability measures.
We note at least $3$ of the $p_i$ must be odd in order to have high rank (see Remark \ref{rmk: 6 odd zeros}), in which case we know that $\cQ$ must be connected \cite{Lanneau, CMexceptional}.

In the following two results, the proof follows along the same lines as in \S\ref{subsec:proof of mainthm}, we just check that we can apply Corollary \ref{cor: high rank QD identification}.

\begin{theorem}\label{thm: plumbingQD}
Consider any gluing of regular plumbing fixtures with boundary lengths $L$, as in Section \ref{subsec:plumbing}.
Suppose that at least 3 of the $p_i$ are odd and that the pair $(\gamma, \mathsf X)$ is not jointly orientable.
Fix any choice of heights $\bfh$ such that $\sum h_i = 1/L$. 
There is a set of times $Z$ of zero density such that
\[
\mu_{\gamma}(e^t\mathsf{X},e^{-t}\bfh)
\xlongrightarrow{t \to \infty, \, t \not \in Z}
\cO^*\nu_{\cQ}^1
\text{ as measures on }
\PoM_g,
\]
where $\nu_{\cQ}^1$ is the Masur--Smillie--Veech measure on $\QoM_g\left((p_1-2)^{2}, \ldots, (p_k-2)^2; -1\right)$.
\end{theorem}
\begin{proof}
Let $p_{\text{odd}}$ and $p_{\text{even}} = k - p_{\text{odd}}$ denote the number of $p_i$ which are odd and even, respectively. Then:
\[\#\gamma = \frac{1}{2}\sum_{i=1}^k p_i,
\qquad
\# \singodd = 2 p_{\text{odd}},
\qquad
\text{and}
\qquad
N_o(\mathsf Y) = p_{\text{even}}.
\]
We also compute that $4g-4 = 2 \sum_{i=1}^k (p_i-2)$, and so we satisfy the hypotheses of the Corollary if and only if
\[\#\gamma \ge 
(\# \gamma - k + 1) + \frac{1}{2}p_{\text{odd}} + p_{\text{even}} + \frac 1 2,\]
which by arithmetic occurs if and only if $p_{\text{odd}} \ge 3$.
\end{proof}

We conclude with a substantial generalization of Theorem \ref{mainthm} which uses far fewer curves.

\begin{theorem}\label{thm:trivalent2g}
Suppose that $\gamma$ is a multicurve on a surface $S$ with $\#\gamma \ge 2g$.
Fix a trivalent $\rg \in \TRG(S \setminus \gamma)$ and pick any $\bfh$ such that $\bfl \cdot \bfh = 1$.
There is a set of times $Z$ of zero density such that
\[\mu_{\gamma}(e^t\rg,e^{-t}\bfh)
\xlongrightarrow{t \to \infty, \, t \not \in Z}
\mu_{\Mirz}/b_g
\text{ as measures on }
\PoM_g.\]
In particular, projecting down to $\M_g$, the twist tori
$\TT_\gamma(e^t \rg)$ for $t \not \in Z$ equidistribute to $\frac{B(X )}{b_g} d \vol_{\WP}(X)$.
\end{theorem}
\begin{proof}
    The condition that $\rg$ is trivalent gives that $\#\sing_{\text{odd}} = 4g-4$, and since no trivalent ribbon graph is orientable, $N_o(\rg)=0$.  Therefore, we have 
    \[g +\frac{1}{4}\#\sing_{\text{odd}} +N_o(\rg)+\frac12 = 2g-\frac12. \]
    Thus, if $\#\gamma\ge 2g$, Corollary \ref{cor: high rank QD identification} identifies the $P$-orbit closure of the corresponding flat twist torus.  
\end{proof}

\bibliography{references}{}
\bibliographystyle{amsalpha.bst}

\Addresses
\end{document}